\documentclass[12pt]{amsart}

\usepackage{amsmath}
\usepackage{amssymb}
\usepackage{amsthm}
\usepackage{mathrsfs}
\usepackage{comment,verbatim}
\usepackage{hyperref}

\usepackage{tikz}
\usepackage{tikz-cd}
\usepackage{rotating}

\usetikzlibrary{matrix}
\usetikzlibrary{arrows}
\usetikzlibrary{decorations.pathmorphing}
\usepackage{wrapfig}

\newcommand\soutpars[1]{\let\helpcmd\sout\parhelp#1\par\relax\relax}
\long\def\parhelp#1\par#2\relax{%
  \helpcmd{#1}\ifx\relax#2\else\par\parhelp#2\relax\fi%
}

\usepackage[all,cmtip]{xy}\usepackage{xcolor}
\usepackage{enumerate}
\usepackage{bm}
\usepackage{dsfont}
\usepackage{mathtools}
\usepackage[cal=euler]{mathalfa}
\usepackage[top=1in, bottom=1.25in, left=1.25in, right=1.25in]{geometry}
\usepackage{parskip}

\hypersetup{colorlinks=true,linkcolor=magenta,citecolor=blue}

\DeclareMathAlphabet{\mathpzc}{OT1}{pzc}{m}{it}

\theoremstyle{plain}
\newtheorem{theorem}{Theorem}[section]
\newtheorem*{theorem*}{Theorem}

\newtheorem{lemma}[theorem]{Lemma}
\newtheorem*{claim*}{Claim}
\newtheorem{proposition}[theorem]{Proposition}

\theoremstyle{definition}
\newtheorem{definition}[theorem]{Definition}

\newtheorem{egr}[theorem]{Example}
\newenvironment{example}{\begin{egr}}{
  \hfill \qedsymbol
  \end{egr}}

\newtheorem{remark}[theorem]{Remark}

\numberwithin{equation}{section}
\numberwithin{figure}{section}

\newcommand{\val}{\mathrm{val}}
\newcommand{\LieG}{\mathfrak{g}}
\newcommand{\LieZ}{\mathfrak{z}}

\newcommand{\Res}{\mathrm{Res}} 

\newcommand{\Cent}{\mathrm{Cent}}

\newcommand{\diag}{\textrm{diag}}
\newcommand{\Norm}{\mathrm{Nm}_{E/F}}  

\newcommand{\Sp}{\mathrm{Sp}}
\newcommand{\GL}{\mathrm{GL}}
\newcommand{\Gal}{\mathrm{Gal}}

\newcommand{\triv}{{\bf 1}}

\newcommand{\Lie}{\mathrm{Lie}}
\newcommand{\LieT}{\mathfrak{t}}
\newcommand{\LieH}{\mathfrak{h}}
\newcommand{\LieS}{\mathfrak{s}}

\newcommand{\bG}{\mathbf{G}}
\newcommand{\bS}{\mathbf{S}}
\newcommand{\bM}{\mathbf{M}}

\newcommand{\bZ}{\mathbf{Z}}

\newcommand{\bH}{\mathbf{H}}
\newcommand{\bT}{\mathbf{T}}

\newcommand{\bU}{\mathbf{U}}

\newcommand{\apart}{\mathscr{A}}

\newcommand{\buil}{\mathscr{B}}
\newcommand{\builred}{\mathscr{B}^{\mathrm{red}}}

\newcommand{\der}{\mathrm{der}}

\newcommand{\red}{\mathrm{red}}

\newcommand{\Gi}[2]{{#1}^{#2}}
\newcommand{\phii}[2]{{#1}_{#2}}
\newcommand{\X}[1]{X_{#1}}
\newcommand{\chara}[1]{\vartheta(#1)}
\newcommand{\type}[1]{\kappa(#1)}
\newcommand{\group}[1]{K(#1)} 
\newcommand{\groupp}[1]{K_+(#1)}
\newcommand{\roi}{\mathfrak{o}_F}
\newcommand{\maxid}{\mathfrak{p}_F}

\newcommand{\singlequasicharacter}{single-quasicharacter}

\begin{document}
\setlength{\parindent}{0pt}
\title[Lifting semisimple characters from fixed-point subgroups]{Lifting semisimple characters of $p$-adic types from fixed-point subgroups}
\author{Ad\`ele Bourgeois}

\address{School of Mathematics and Statistics, Carleton University, Ottawa, Canada}
\email{abour115@uottawa.ca}

\author{Monica Nevins}
\address{Department of Mathematics and Statistics, University of Ottawa, Ottawa, Canada}
 
\email{mnevins@uottawa.ca}
\thanks{The second author's research was supported NSERC Discovery Grant RGPIN-2025-05630 and the Institut Henri Poincar\'e (UAR 839 CNRS-Sorbonne Universit\'e) grant number ANR-10-LABX-59-01.}
\date{\today}

\begin{abstract}
\noindent Given a $p$-adic group $G=\bG(F)$ and a finite group $\Gamma\subset\mathrm{Aut}_F(\bG)$ such that the fixed-point subgroup $\bG^\Gamma$ is reductive, we show that every semisimple character (in the sense of Bushnell and Kutzko) of a type for $G^\Gamma = \bG^\Gamma(F)$ arises as the restriction of a semisimple character of a type for $G$. We achieve this by explicitly lifting the truncated Kim--Yu datum (or character-datum) that parametrizes the semisimple character for $G^\Gamma$ to a character-datum that parametrizes a semisimple character for $G$. Our proof, which is of independent interest, uses state-of-the-art techniques and, as a special case, defines a lift of a Howe factorization of a character of a maximal torus of $G^\Gamma$.
\end{abstract}

\maketitle


\section{Introduction}

The theory of types, for a $p$-adic group $G$, realizes the Bernstein decomposition  of the category of smooth representations of $G$ into blocks: the irreducible representations in each block are characterized by the types $(K,\rho)$ they contain upon restriction, where $K$ ranges over certain compact open subgroups and $\rho$ are particular representations of $K$.  This theory is proven, and all types have been constructed, in many cases, including notably \cite{BushnellKutzko1993, Stevens2008, Yu2001,Kim2007,KimYu2017,Fintzen2021fix, Fintzen2021}.

The known constructions of types for positive-depth irreducible representations of a $p$-adic group $G$ involve two pieces of data:  a depth-zero part, which employs cuspidal representations of the reductive quotients of parahorics of $G$; and a positive-depth part, which is built from suitable quasicharacters of twisted Levi subgroups of $G$.  The output of the positive-depth part is a quasicharacter $\vartheta$ of a pro-$p$ subgroup $K_+$ which in \cite[Ch 3]{BushnellKutzko1993} was termed a \emph{semisimple character} of a type for $G$.

Now let $\bG$ be a reductive group over a nonarchimedean local field $F$ of residual characteristic $p$, and suppose $\Gamma$ is a finite group of automorphisms of $\bG$ of order prime to $p$.  Then $\bG^{[\Gamma]}:=(\bG^{\Gamma})^\circ$ is a reductive subgroup \cite{PrasadYu2002}. 
Prior work has shown that there exists a class of semisimple characters of types for $G=\bG(F)$, called $\Gamma$-stable in \cite{LathamNevins2023}, self-dual in \cite{MiyauchiStevens2014} and skew in \cite{Stevens2005}, whose restrictions to $G^{[\Gamma]}:=\bG^{[\Gamma]}(F)$ are semisimple characters of types for $G^{[\Gamma]}$.  This was a key step in classifying the supercuspidal representations of classical groups in \cite{Stevens2005,Stevens2008, MiyauchiStevens2014} and for $G_2$ in \cite{BlascoBlondel2012}.  

In this paper, we prove under mild restrictions on $p$ (see Section~\ref{sec:prelims}), for general connected reductive groups  $\bG$ that split over a tamely ramified extension of $F$, 
that conversely, every semisimple character of a type for $G^{[\Gamma]}$ is the restriction of a $\Gamma$-stable semisimple character of a type for $G$, answering an open question in \cite{LathamNevins2023}.    Our solution is in fact an explicit construction of such a lift, obtained as a consequence of a deep dive into the lifting of the sequence of quasicharacters.

Let us present our results in more detail.    When $p$ does not divide the order of the Weyl group of $\bG$ and $\bG$ splits over a tamely ramified extension of $F$, Fintzen showed that every smooth irreducible representation of $G$ contains a \emph{Kim--Yu type} \cite[Theorem 7.12]{Fintzen2021}. The construction of Kim--Yu types is in terms of Kim--Yu data \cite[\S7]{KimYu2017}. Omitting the depth-zero part of a Kim--Yu datum gives a \emph{character-datum} for $G$ (Definition \ref{def:chardatum}), which is precisely the input needed to construct the semisimple character of a Kim--Yu type. Thus, the study of semisimple characters of Kim--Yu types can be further reduced to the study of character-data for $G$.

 Our main theorem can be summarized as follows.

\begin{theorem*}[Theorem \ref{th:lift}]
Let $\Gamma \subset \mathrm{Aut}_F(\bG)$ be of order prime to $p$, and let $\bG^{[\Gamma]} = (\bG^\Gamma)^\circ$. Let $\Delta$ be a character-datum for $G^{[\Gamma]} = \bG^{[\Gamma]}(F)$ and denote its corresponding semisimple character by $\chara{\Delta}$, which is a character of a compact open subgroup $\groupp{\Delta} \subset G^{[\Gamma]}$. Then there exists a $\Gamma$-stable character-datum $\Sigma$ for $G$ such that $\groupp{\Sigma}^\Gamma = \groupp{\Delta}$ and $\chara{\Delta} = \chara{\Sigma}|_{\groupp{\Delta}}.$
\end{theorem*}

In the course of proving this theorem, we slightly generalize the definition of a $\Gamma$-stable character-datum from \cite{LathamNevins2023} to suit our needs.  In \cite{LathamNevins2023} it is shown that such a character-datum $\Sigma$ for $\bG$ defines a character-datum $\Sigma^\Gamma$ for $\bG^{[\Gamma]}$ such that the corresponding semisimple characters are compatible via restriction.  To prove Theorem~\ref{th:lift} we thus construct a character-datum $\Sigma$ such that $\Sigma^\Gamma$ is a \emph{refactorization} of $\Delta$, in the sense of \cite{HakimMurnaghan2008}.

Let us sketch the proof, 
which is entirely constructive. 

The character-datum for $G^{[\Gamma]}$ contains a sequence of generic quasicharacters of twisted Levi subgroups of $\bG^{[\Gamma]}$.  In Theorem~\ref{T:liftone}, we establish how to lift a \emph{single} $G^{[\Gamma]}$-generic quasicharacter $\xi$ of a twisted Levi subgroup $H'=\bH'(F) \subset G^{[\Gamma]}$ to a sequence of generic quasicharacters of twisted Levi subgroups of $\bG$.
The proof of this theorem relies heavily on techniques from \cite{Fintzen2021} for constructing characters and twisted Levi subgroups in a controlled fashion (Theorem \ref{T:charexists} and Proposition \ref{P:indstep}). 

To lift a character-datum $\Delta$ for $G^{[\Gamma]}$, one must lift all its quasicharacters in a compatible way, while ensuring that the restriction $\Sigma^\Gamma$ of the obtained character-datum $\Sigma$ for $G$ is a refactorization of $\Delta$. As such, we cannot naively lift all quasicharacters in our sequence independently, but rather require an iterative construction which uses a top-down approach.  This argument occupies most of Section~\ref{sec:completelift}. 

We illustrate the obstacles to naive lifting throughout with examples to help motivate the necessity of this technically-challenging construction. The steps for the proof of Theorem \ref{th:lift} are illustrated in Figure~\ref{fig:roadmap}.

\begin{figure}[!htbp]
\begin{center}
\begin{tikzcd}[row sep=0, column sep=-6.5em]
{} &{} &\text{Theorem \ref{th:lift}} &{}  \\
{} &{} &{} &{} \\
{} &{} &{} &{} \\
{} &\parbox{5.5cm}{\centering Lifting a single quasicharacter (Theorem \ref{T:liftone})}\arrow{uuur} &{} &\parbox{4.5cm}{\centering Top-down\\ iteration (Section~\ref{sec:completelift})}\arrow{uuul} \\
{} &{} &{} &{}  \\
{} &{} &{} &{} \\
\parbox{6.5cm}{\centering Constructing characters with specified restriction (Theorem \ref{T:charexists})}\arrow{uuur} &{} &\parbox{5.5cm}{\centering Constructing $\Gamma$-stable twisted Levis (Proposition \ref{P:indstep})}\arrow{uuul} &{} 
\end{tikzcd}
\end{center}
\caption{\label{fig:roadmap}Roadmap for proving Theorem \ref{th:lift}.}
\end{figure}
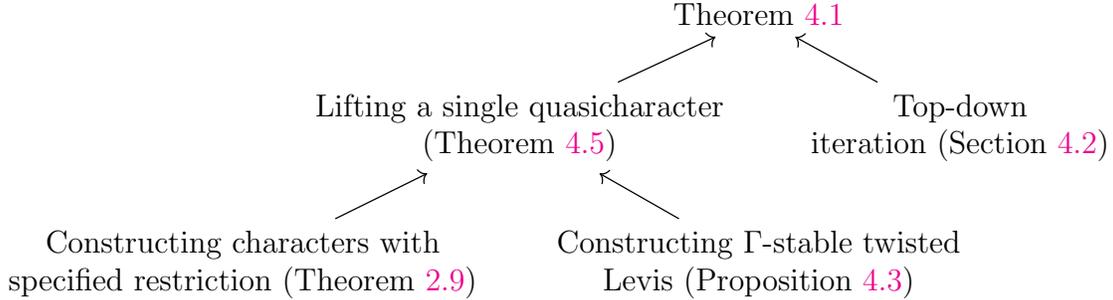

A character-datum for $G^{[\Gamma]}$ will admit many lifts to $\Gamma$-stable character-data for $G$
and our construction produces all of them.  In rough terms: at each step in the inductive construction, one chooses an element $X^\sharp\in \LieG^{\Gamma,*}$ and then a  quasicharacter $\phi'$ of $\Cent_{G}(X^\sharp)$, with particular properties, whose existence  are proven in Proposition~\ref{P:indstep} and Theorem~\ref{T:charexists}, respectively.   
In Propositions~\ref{prop:resthenliftsingle} and \ref{prop:resthenlift}, we prove that if $\Sigma$ is a $\Gamma$-stable character-datum for $G$, and $\Delta = \Sigma^\Gamma$, then there exists a set of choices of $(X_i^\sharp,\phi_i')_i$ that yields $\Sigma$ as a lift of $\Delta$.

Theorem~\ref{T:liftone} slightly generalizes the notion of Howe factorization of quasicharacters of tori of \cite[Definition 3.6.2]{Kaletha2019} to quasicharacters of twisted Levi subgroups, and makes the construction of $G$-factorizations \cite[Definition 5.1]{Murnaghan2011} explicit. We explore these connections in more detail in Section~\ref{sec:otheraves}.

The question studied in this paper is part of a larger story.  The tight correspondence of semisimple characters, and of associated character-data, has also been established for pairs $(\bG,\bH)$ where $\bH$ is a subgroup of $\bG$ containing $\bG_{\mathrm{der}}$ in 
\cite{Bourgeois2021}. That is, one can construct a character-datum for $H$ from a character-datum for $G$, and vice versa, and relate the corresponding semisimple characters via restriction. Importantly, this has a direct application to functoriality in the Langlands correspondence \cite{BourgeoisMezo}.  

Similarly, twisted endoscopy predicts a relationship between the semisimple characters for $G$ and groups such as $H = G^{[\Gamma]}$.  However, the motivation from base change suggests that instead of simple restriction, the semisimple character $\vartheta_G$ should be determined by averaging $\vartheta_H$ over each $\Gamma$-orbit, in the sense that $\vartheta_H\circ \prod_{\gamma\in \Gamma}\gamma=\vartheta_G$.  This is the subject of ongoing work.

The paper is organized as follows. Section \ref{sec:prelims} contains preliminaries, beginning with notation and conventions (Section \ref{sec:notation}), followed by a summary of important results concerning genericity (Section \ref{sec:genericity}), and ending with the construction of semisimple characters for Kim--Yu types from character-data for $G$ (Section \ref{sec:semisimplechars}). Section \ref{sec:restrict} revisits the restriction results from \cite{LathamNevins2023}, and introduces a generalized version of the main results therein. We prove our main theorem (Theorem \ref{th:lift}) in Section \ref{sec:lift}, illustrating the main steps with  examples. Section \ref{sec:liftchar} addresses the lifting of a single quasicharacter, and then Section \ref{sec:completelift} combines all the intermediate results of this paper to describe an iterative construction for lifting a character-datum for $G^{[\Gamma]}$ to a character-datum for $G$. We discuss connections with related constructions in Section~\ref{sec:otheraves}.

\subsection*{Acknowledgments} The authors thank Jeff Adler and Loren Spice for many motivating discussions. 
Some of this work was completed while the second author participated in the IHP (Paris) thematic trimester programme \emph{Representation Theory and Noncommutative Geometry} from January to April, 2025.
  
\section{Preliminaries}\label{sec:prelims}

\subsection{Notation and conventions}\label{sec:notation}

Given the nonarchimedean local field $F$, we denote by $\mathrm{val}(\cdot)$, its normalized valuation
; by $\roi$, its ring of integers; by $\maxid$, the unique maximal ideal of $\roi$; and by $\mathfrak{f}$, its residue field of prime characteristic $p$. We fix a choice of additive character $\psi:F\rightarrow \mathbb{C}^\times$ that is trivial on $\maxid$ but non-trivial on $\roi$.

Unless otherwise specified, $\bG$ denotes a connected reductive group that is defined over $F$. By a twisted Levi subgroup of $\bG$ we mean a subgroup of $\bG$ defined over $F$ that is a Levi subgroup over some finite extension of $F$. We set $G = \bG(F)$ and $\bG_{\der} = [\bG,\bG]$, the derived group. We also set $\mathrm{Lie}(\bG)$ to be the Lie algebra of $\bG$, $\LieG = \mathrm{Lie}(\bG)(F)$ and $\LieG^*$ to be the dual of $\LieG$. We use analogous notational conventions (for example, $M$, $\mathfrak{m}$, $\mathfrak{m}^*$) for an algebraic $F$-subgroup $\bM$ of $\bG$ without comment. 

We will assume $\bG$ splits over a tamely ramified extension of $F$, and that $p$ is coprime to the order of the (absolute) Weyl group of $\bG$. This hypothesis ensures that every twisted Levi subgroup of $\bG$ that is defined over $F$ splits over a tamely ramified extension \cite[\S2]{Fintzen2021}.

We write $\buil(G)$ (respectively, $\builred(G)$) for the enlarged (respectively, reduced) Bruhat--Tits building of $G$. For each $x\in\buil(G)$, we let $[x]$ denote the projection of $x$ in $\builred(G)$. For each $x\in \buil(G)$ and $r>0$, $G_{x,r}$ denotes the corresponding Moy--Prasad filtration subgroup of the parahoric subgroup $G_{x,0}$. We also set $G_{x,r^+} = \underset{t > r}{\cup}G_{x,t}$ for all $r\geq 0$, and recall that for each $x$ there exist values $r'>r$ such that $G_{x,r^+} = G_{x,r'}$. For any $0 < t < r$ we abbreviate $G_{x,t:r} = G_{x,t}/G_{x,r}$. 

Similarly, for any $r\in\mathbb{R}$ and $x\in\buil(G)$, we have a filtration of $\LieG$ by $\roi$-lattices $\LieG_{x,r}$ and $\LieG_{x,r^+}$. This induces a corresponding filtration on $\LieG^*$ by setting $${\LieG^*_{x,-r} = \{\X{} \in \LieG^* : \X{}(\LieG_{x,r^+})\subseteq \maxid\}}$$  and  $\LieG^*_{x,(-r)^+} = \underset{r' < r}{\cup}\LieG^*_{x,-r'}$. For any $0 < t < r \leq 2t$ and $x\in\buil(G)$, set $\LieG_{x,t:r} = \LieG_{x,t}/\LieG_{x,r}$. Then under our tameness hypothesis we have a Moy--Prasad isomorphism of abelian groups $e_{x,t:r}: \LieG_{x,t:r} \to G_{x,t:r}$; see \cite[Theorem 13.5.1]{KalethaPrasad}. 

Given a maximal $F$-torus $\bT$ of $\bG$, we write $\Phi(\bG,\bT)$ for the corresponding root system. For each $\alpha\in\Phi(\bG,\bT)$, we let $\bU_\alpha$ denote the corresponding root subgroup and let $H_\alpha = d\alpha^\vee(1)$ denote the corresponding coroot. Given $X\in \LieT^* = \mathrm{Lie}(\bT)(F)^*$, we recall that $\Cent_\bG(X) = \langle \bT, \bU_\alpha : X(H_\alpha) = 0 \rangle$ is a Levi subgroup of $\bG$ \cite[Propositions 7.1 and 7.2]{Yu2001}.

We fix a finite subgroup $\Gamma \subset \mathrm{Aut}_F(\bG)$, whose order is coprime to $p$. Given a $\Gamma$-stable subgroup $\bH$ of $\bG$, we write $\bH^{[\Gamma]} = (\bH^\Gamma)^\circ$ for the connected component of the fixed-point subgroup and set $H^{[\Gamma]} = \bH^{[\Gamma]}(F)$. By \cite[Theorem 2.1]{PrasadYu2002}, $\bG^{[\Gamma]}$ is a closed reductive subgroup of $\bG$. Furthermore, $\buil(G^{[\Gamma]}) = \buil(G)^\Gamma$ \cite[Proposition 3.4]{PrasadYu2002} so that $\buil(G^{[\Gamma]})\hookrightarrow \buil(G)$. Note, however, that when the maximal split torus in the center of $\bG^{[\Gamma]}$ is not central in $\bG$, it will not be true that $\builred(G)^\Gamma = \builred(G^{[\Gamma]})$. We also have a natural identification $\mathrm{Lie(\bG^{[\Gamma]})^*} = (\mathrm{Lie}(\bG)^*)^\Gamma$ \cite[Lemma 3.4]{LathamNevins2023}.


Finally, we will also assume that $\bG^{[\Gamma]}$ splits over a tamely ramified extension of $F$, and that $p$ is coprime to the order of the (absolute) Weyl group of $\bG^{[\Gamma]}$.

\subsection{Realizing characters and genericity}\label{sec:genericity} 

The construction of semisimple characters requires quasicharacters that are \emph{generic} \cite{Yu2001}. In this section, we recall this notion and several equivalent characterizations that are needed in the following sections. We also prove  Theorem~\ref{T:charexists}, concerning the existence of quasicharacters with a specified restriction condition, using techniques in \cite{Fintzen2021}. 
This theorem will be a key ingredient for lifting quasicharacters of fixed-point subgroups in Section \ref{sec:liftchar}.
 
Recall that $\psi$ is a fixed non-trivial additive character of $F$ of depth zero.  For the remainder of this section, let $\bG'$ denote a twisted Levi subgroup of $\bG$ that is defined over $F$, and contains a maximal $F$-torus $\bT$. By assumption, $\bT$ splits over a tamely ramified extension $E$ of $F$. Let $\Phi = \Phi(\bG,\bT)$ and $\Phi' = \Phi(\bG',\bT)$ be the corresponding absolute root systems.

Set $\bZ' = Z(\bG')^\circ$ and $\LieZ' = \mathrm{Lie}(\bZ')(F)$, so that
$$
\LieZ^{'*}_{-r} = \{\X{} \in \LieZ^{'*} : \X{}(\LieZ'_{r^+})\subseteq \mathfrak{p}_F\}
$$
and $\LieZ^{'*}_{(-r)^+} = \underset{r' < r}{\cup}\LieZ^{'*}_{-r'}$ for all $r\in \mathbb{R}$. One can view an element of $\LieZ^{'*}$ as an element of $\LieT^{*}$ or $\LieG^{'*}$ by extending it trivially on $\LieG'_{\der}=\mathrm{Lie}(\bG'_\der)(F)$ \cite[Proposition 3.1]{AdlerRoche2000}. Similarly, for any $r\in \mathbb{R}$ and $x\in \buil(G')$, an element of $\LieZ^{'*}_{-r}$ can be viewed as an element of $\LieT^{*}_{-r}$ or $\LieG^{'*}_{x,-r}$ 
\cite[Proposition 3.2]{AdlerRoche2000}.  We can also extend scalars to $E$, meaning that elements of $\LieZ^{'*}$ can be evaluated on $\mathrm{Lie}(\bG')(E)$. 

Suppose $X\in \LieZ^{'*}_{-r}$ and $x\in \buil(G')$, and let  $s \in (r/2, r]$.  Then $X$ defines a quasicharacter $\eta_{s,\X{}}$ of $G'_{x,s}/G'_{x,r^+}$ via the Moy--Prasad isomorphism by the rule 
$$
\eta_{s,X}(e_{x,s:r^+}(Y))=\psi(X(Y))
$$ 
for all $Y\in \LieG'_{x,s}$.  We similarly define $\eta_{s^+,X}$ for any $s\in [r/2,r)$. The following is a variation on \cite[Lemma 2.51]{HakimMurnaghan2008}. 

\begin{lemma}\label{L:etasXdependsoncoset}
Let $x\in \buil(G')$, $r>0$ and $s\in [r/2,r)$. Two elements $X,X'\in \LieZ^{'*}_{-r}$ satisfy $\eta_{s^+,X}=\eta_{s^+,X'}$ on $G'_{x,s^+}$ if and only if $X\in X'+\LieZ^{'*}_{-s}$, or equivalently, if $X\in X'+\LieG^{'*}_{x,-s}$.  In particular, $\eta_{s^+,X}|_{G'_{x,r}} = \eta_{r,X}$ is trivial if and only if $X\in \LieZ^{'*}_{(-r)^+}$.
\end{lemma}

\begin{proof}
We have $\eta_{s^+,X}=\eta_{s^+,X'}$ if and only if $(X-X')(Y) \in \ker\psi$ for all $Y\in\LieG'_{x,s^+}$. Following the proof of \cite[Proposition 3.1.12]{BourgeoisThesis}, linearity implies this is equivalent to $(X-X')(Y) \in \mathfrak{p}_F$ for all $Y\in\LieG'_{x,s^+}$, or equivalently, $X-X'\in \LieG^{'*}_{x,-s}$.  Since $X,X'\in \LieZ'$ the first result follows. 
For the second statement, the compatibility of the Moy--Prasad isomorphisms \cite[Remark 3.4.3]{Hakim2018} ensures $\eta_{s^+,X}|_{G'_{x,t^+}} = \eta_{t^+,X}$ for every $s\leq t < r$.  Choosing $t$ 
such that $G_{x,t^+} = G_{x,r}$, we have $\LieG^{'*}_{x,-t}=\LieG^{'*}_{x,(-r)^+}$ and $\eta_{r,X}=\eta_{t^+,X}$. Then $\eta_{r,X}$ is trivial if and only if $\eta_{t^+,X} = \eta_{t^+,0}$. The conclusion follows.
\end{proof}

\begin{definition}\label{def:realizedby}
Let $\phi$ be a character of $G'$ of depth $r$, $x\in\buil(G')$, $\X{}\in\LieZ^{'*}_{-r}$ and $s\in (r/2,r]$. We say that $\phi$ is realized by $\X{}$ on $G'_{x,s}$ if $\phi|_{G'_{x,s}}=\eta_{s,X}$.
\end{definition}

\begin{remark}\label{rem:realizedcenter}
Definition \ref{def:realizedby} is not restrictive (under our tameness hypothesis).  Indeed, 
given a  character $\phi$ of $G'$ of depth $r$ and $x\in\buil(G')$, the following hold.
\begin{enumerate}
\item[1)] By \cite[Lemma 3.5.2]{Kaletha2019}, $\phi$ is realized by an element $\X{}\in\LieZ^{'*}_{-r}$ on $G'_{x,r/2^+}$.
\item[2)] Since the Moy--Prasad isomorphisms 
for different $s\in (r/2,r]$ are naturally compatible \cite[Remark 3.4.3]{Hakim2018}, if $\X{}\in\LieZ^{'*}_{-r}$ realizes $\phi$ on $G'_{x,s}$, then $\X{}$ also realizes $\phi$ on $G'_{x,r}$.
\item[3)] By \cite[Lemma 3.5.2]{Kaletha2019} and \cite[Lemma 4.7]{Murnaghan2011}, if $\X{}\in \LieZ^{'*}_{-r}$ realizes $\phi$ on $G'_{x,r}$, then $\X{}$ realizes $\phi$ on $G'_{y,r}$ for all $y\in\buil(G')$. 
\end{enumerate}
\end{remark}

Let $r\in\mathbb{R}$ and $\X{} \in \LieT^{*}_{-r}$.  For any $\alpha\in \Phi$ we have  $H_\alpha\in \LieT(E)_0$ for some tame extension $E$ of $F$.  It follows that  $\val(\X{}(H_\alpha)) \geq -r$ for all $\alpha\in \Phi$. In particular this applies to elements of $\LieZ^{'*}_{-r}$.


\begin{definition}\label{def:genericel}
Let $r\in \mathbb{R}$ and $\X{} \in \LieZ^{'*}_{-r}\setminus \LieZ^{'*}_{(-r)^+}$. If $\val(\X{}(H_\alpha)) = -r$ for all $\alpha\in \Phi\setminus\Phi'$, then $\X{}$ is said to be $G$-generic of depth $-r$.
\end{definition}

\begin{remark}\label{rem:genericgeneral}
It is possible to define genericity without first specifying  $\bG'$.  An element $X\in \LieG^*$ is called almost stable of depth $-r$ if there exists a maximal torus $\bT$ such that $X\in \LieT^*_{-r}\smallsetminus \LieT^*_{(-r)^+}$. Then following \cite[Definition 3.5]{Fintzen2021} we say such an $X$ is $G$-generic of depth $-r$ if for all $\alpha \in \Phi(\bG,\bT)$, either $X(H_\alpha)=0$ or $\val(X(H_\alpha))=-r$.  In this case, setting $\bG'=\Cent_{\bG}(X)$, we infer that $X$ is $G$-generic of depth $-r$ in the sense of Definition~\ref{def:genericel}.  The  genericity of $X$ as an element of $\LieZ^{'*}$ is independent of the choice of $\bT$ (see also \cite[\S5]{AdlerRoche2000}).  Furthermore, by \cite[Corollary 3.8]{Fintzen2021}, such an $X$ is \emph{almost strongly stable} in the sense of \cite[Definition 3.1]{Fintzen2021}.
\end{remark}

We are now ready to state the definition of genericity for a quasicharacter of $G'$.

\begin{definition}\label{def:genericchar}
Let $\phi$ be a non-trivial quasicharacter of $G'$ of depth $r$. We say that $\phi$ is $G$-generic 
if, for every $x\in \buil(G')$, $\phi$ is realized on $G'_{x,r}$ by a $G$-generic element $\X{}\in\LieZ^{'*}_{-r}$.
\end{definition}

We note that when $\bG' = \bG$, genericity of $X\in \LieZ^*$ is simply the condition that it has depth $-r$, and thus every quasicharacter of $G'$ is $G$-generic.

\begin{remark}\label{rem:genexplained}
By \cite[Remark 6.4]{Murnaghan2011}, if a quasicharacter $\phi$ of $G'$ is $G$-generic of depth $r$, then $\phi|_{G'_{x,r}}$ cannot be obtained as the restriction of a quasicharacter of a larger twisted Levi subgroup of $G$. 
\end{remark}

\begin{lemma} \label{lem:realizedindependence}
Let $\phi$ be a quasicharacter of $G'$ of depth $r$. Then its genericity is independent of the choices of $x\in \buil(G')$, and $G$-generic $X\in \LieZ^{'*}_{-r}$ realizing $\phi$ on $G_{x,s}$ for any $s\in (r/2,r]$.
\end{lemma}

\begin{proof}
From Definition~\ref{def:genericel} and Lemma~\ref{L:etasXdependsoncoset} it follows that the  genericity of $X\in \LieZ^{'*}_{-r}\setminus \LieZ^{'*}_{(-r)^+}$ is a property of the coset $X+\LieZ^{'*}_{(-r)^+}$.  By Remark~\ref{rem:genexplained}(3), $X$ is independent of the choice of  $x\in \buil(G')$ at which we realize $\phi$.  Finally, if $X'\in \LieZ^{'*}_{-r}$ realizes $\phi$ on $G_{x,s}$ for some $s\in (r/2,r]$, then from Remark~\ref{rem:genexplained}(2) we conclude $X'\in X+ \LieZ^{'*}_{(-r)^+}$.
\end{proof}
\color{black}

Our final result applies to every $\eta_{s,X}$ as defined above, independent of the genericity of $X$, showing it can be extended to a character $\phi$ of $G'$. This will be a key ingredient for lifting (not necessarily $G$-generic) quasicharacters of fixed-point subgroups in Section~\ref{sec:liftchar}.

\begin{theorem} \label{T:charexists}
    Let $\bG'$ be a twisted Levi subgroup of $\bG$ and let $x\in \buil(G')$.  Set $\mathfrak{z}'=\Lie(Z(\bG')^\circ)(F)$ and  let $\X{}\in {\mathfrak{z}'}_{-r}^*\setminus {\mathfrak{z}'}_{(-r)^+}^*$ for some $r>0$.  Fix some $s \in (r/2, r]$ and let $\eta_{s,\X{}}$ be the character of $G'_{x,s}/G'_{x,r^+}$ realized by $\X{}$.  Then there exists a quasicharacter $\phi$ of $G'$ of depth $r$ whose restriction to $G'_{x,s}$ coincides with $\eta_{s,\X{}}$.  
\end{theorem}

\begin{proof}
 Let $\X{} \in {\mathfrak{z}'}_{-r}^*\setminus {\mathfrak{z}'}_{(-r)^+}^*$. 
    Let $\bT$ be a maximal $F$-split torus of $\bG'$ whose corresponding affine apartment $\apart(\bG',\bT,F)$ contains $x$.  
    Recall that we have identified $\LieZ^{'*}$ with the subspace of $\LieG^{'*}$ consisting of linear functionals that vanish on $\LieG'_{\der}$. 
    Since the Moy--Prasad isomorphism respects the decomposition $\LieG'=\LieZ'\oplus \LieG'_{\der}$ 
    \cite[Proposition 3.2]{AdlerRoche2000}, $\eta_{s,\X{}}$ is trivial on $(G'_{\der})_{x,s}$.  Let $\overline{\eta_{s,\X{}}}$ be the character of the group 
    $$T_s(T_{0^+}\cap G'_{\der})=T_s(T\cap G'_{\der})_{0^+}$$ 
    that trivially extends the restriction of $\eta_{s,\X{}}$ to $T_s$.  That is, since $\eta_{s,\X{}}$ is trivial on $(T\cap G'_{\der})_s$, we may define $\overline{\eta_{s,\X{}}}$ on each $t\in T_s$ and $u\in (T\cap G'_{\der})_{0^+}$ by  $\overline{\eta_{s,\X{}}}(tu)=\eta_{s,\X{}}(t)$.  

    Since $T_{0^+}$ is an abelian group, $\overline{\eta_{s,\X{}}}$ can be extended (nonuniquely) to a character $\varphi$ of $T_{0^+}$ that is trivial on $(T\cap G'_{\der})_{0^+}$.  Note that $\varphi$ is trivial on $T_{r^+}$ by construction, but since $\eta_{s,X}$ has depth $r$, 
    it is nontrivial on $T_{r}$.
    
    Since $p$ does not divide the order of the fundamental group of $G'_{\der}$, the argument in \cite[Proof of Corollary 7.2]{Fintzen2021} may be applied, yielding the isomorphism
    \begin{equation}\label{E:Fintzenexact}
    (\bG'/\bG'_{\der})(F)_{0^+} \cong (\bT/\bT\cap \bG'_{\der})(F)_{0^+} \cong T_{0^+}/(T\cap G'_{\der})_{0^+}.
    \end{equation}
Lift the character $\varphi$ via these isomorphisms to a character $\varphi'$ of $(\bG'/\bG'_{\der})(F)_{0^+}$.  Again, since the algebraic group $\bG'/\bG'_{\der}$ is abelian, $\varphi'$ can be extended (nonuniquely) to a character $\overline{\varphi'}$ of $(\bG'/\bG'_{\der})(F)$.  As $G'/G'_{\der}$ embeds in $(\bG'/\bG'_{\der})(F)$, restricting $\overline{\varphi'}$ to this image gives  a character $\phi$ of $G'$. 

By \cite[Proof of Corollary 7.2]{Fintzen2021}, we have that $G'_{x,r^+}/(G'_{\der})_{x,r^+} \subseteq (\bG'/\bG'_{\der})(F)_{r^+}$.  
 By \cite[Lemma 3.1.3]{Kaletha2019}, \eqref{E:Fintzenexact} holds with the filtration level $0$ replaced by any positive real number.  We may thus deduce that  $\phi$ has depth $r$ and is realized on $G'_{x,s}$ by $\X{}$, as required. 
\end{proof}

The quasicharacter $\phi$ produced by Theorem \ref{T:charexists} is far from unique; if $\bG'$ has noncompact center then there will even be infinitely many choices.

\subsection{Semisimple characters}\label{sec:semisimplechars}

In this section, we recall how to construct semisimple characters of Kim--Yu types.

\begin{definition}\label{def:chardatum}
A character-datum for $G$ is a sequence $\Sigma = (\vec{\bG},y,\vec{r},\vec{\phi})$ where
\begin{enumerate}
\item[CD1)] $\vec{\bG} = (\Gi{\bG}{0}\subsetneq\dots\subsetneq\Gi{\bG}{d}\subseteq \Gi{\bG}{d+1} = \bG)$ is a sequence of twisted Levi subgroups, each of which is defined over $F$ (and splits over a tamely ramified extension of $F$);
\item[CD2)] $y$ is a point in $\buil(\Gi{G}{0}) \subset \buil(\Gi{G}{1})\subset \cdots \subset \buil(G)$, relative to a fixed choice of embeddings of buildings; 
\item[CD3)] $\vec{r} = (r_0,\dots,r_d)$ is a sequence of real numbers satisfying $0 <r_0 <\cdots < r_{d-1} < r_d$. We also set $s_i = r_i/2$ for all $0\leq i\leq d$;
\item[CD4)] $\vec{\phi} = (\phii{\phi}{0},\dots,\phii{\phi}{d})$ is a sequence of quasicharacters, where for each $0\leq i\leq d$,  $\phii{\phi}{i}$ is a $\Gi{G}{i+1}$-generic quasicharacter of $\Gi{G}{i}$ of depth $r_i$ in the sense of Definition~\ref{def:genericchar}. 
\end{enumerate}
\end{definition}

\begin{remark}\label{rem:genericrealize}
For all $0\leq i\leq d$, set $\Gi{\bZ}{i} = Z(\Gi{\bG}{i})^\circ$ and $\Gi{\LieZ}{i} = \mathrm{Lie}(\Gi{\bZ}{i})(F)$. By  
Lemma~\ref{lem:realizedindependence}, 
condition (CD4) is equivalent to saying that, for all $x\in \buil(\Gi{G}{i})$, there exists $\X{i} \in \LieZ^{i*}_{-r_i}$, $\Gi{G}{i+1}$-generic of depth $-r_i$, that realizes $\phii{\phi}{i}$ on $\Gi{G}{i}_{x,s_i^+}$.  A \emph{truncated datum for $G$} as defined in \cite[Definition 2.1]{LathamNevins2023} is a character-datum together with a corresponding sequence of such choices $\X{i}$. In Section~\ref{sec:restrict} we show that the simpler notion of a character-datum suffices.
\end{remark}

Given a character-datum $\Sigma = (\vec{\bG},y,\vec{r},\vec{\phi})$, we construct a compact open  subgroup $\groupp{\Sigma}\subset G$ and character $\chara{\Sigma}$ thereof as follows. We set 
$$
\groupp{\Sigma} = \Gi{G}{0}_{y,0^+}\Gi{G}{1}_{y,s_0^+}\cdots \Gi{G}{d}_{y,s_{d-1}^+}.
$$
For each $0\leq i\leq d$, choose $\X{i}\in \LieZ^{i*}_{-r_i}$ realizing $\phi_i$ on $\Gi{G}{i}_{y,s_i^+}$ and use this to define a character $\eta_{s_i^+,\X{i}}$ of $G_{y,s_i^+}$, as in Section~\ref{sec:genericity}.  Then we define a quasicharacter  $\widehat{\phii{\phi}{i}}$ of $\Gi{G}{i}_yG_{y,s_i^+}$ by 
$$
\widehat{\phii{\phi}{i}}(gh) = \phii{\phi}{i}(g)\eta_{s_i^+,\X{i}}(h) \text{ for all } g\in \Gi{G}{i}_y, h\in G_{y,s_i^+}.
$$ 
Finally, we set 
$$
\chara{\Sigma} = \prod_{i=0}^d \left(\widehat{\phii{\phi}{i}}|_{\groupp{\Sigma}}\right).
$$

A character-datum $\Sigma$ can be completed to a datum $\tilde{\Sigma}$ in the sense of \cite[Section 7.2]{KimYu2017} by letting $\Gi{\bM}{0}$ denote the Levi subgroup of $\Gi{\bG}{0}$ associated to the point $y$ by \cite[Section 6.3]{MoyPrasad1996} and adding a depth-zero datum in the sense of \cite[Section 7.1]{KimYu2017}.

From $\tilde{\Sigma}$, Kim and Yu construct a group $\group{\tilde{\Sigma}}$ and a representation $\type{\tilde{\Sigma}}$ thereof, and show that it is a type \cite{KimYu2017}. Then $\groupp{\Sigma}$ is a pro-$p$ subgroup which is normal in $\group{\tilde{\Sigma}}$, and $\type{\tilde{\Sigma}}|_{\groupp{\Sigma}}$ is $\chara{\Sigma}$-isotypic. Thus, we refer to $(\groupp{\Sigma},\chara{\Sigma})$ as the \emph{semisimple character} associated to the type $(\group{\tilde{\Sigma}},\type{\tilde{\Sigma}})$. For simplicity, we omit mention of the type and say that $(\groupp{\Sigma},\chara{\Sigma})$ is a semisimple character for $G$.

Two character-data may give rise to the same semisimple character.  For example, we may choose another base point $y$ with the same image in $\buil^{\red}(G)$. A more subtle equivalence arises from the notion of \emph{refactorization}, as introduced by Hakim and Murnaghan in \cite{HakimMurnaghan2008}. 
Let us recall this important notion, slightly adjusted for application to character-data.

\begin{definition}[{\cite[Definition 4.19(F1)]{HakimMurnaghan2008}}]\label{def:refactorization}
Let $\Sigma = (\vec{\bG},y,\vec{r},\vec{\phi})$ and $\Sigma' = (\vec{\bG},y,\vec{r},\vec{\phi'})$ be two character-data for $G$. We say that $\vec{\phi'}$ is a refactorization of $\vec{\phi}$ (or that $\Sigma'$ is a refactorization of $\Sigma$) if, for all $0\leq i\leq d$, we have
$$\prod_{j=i}^d \phii{\phi}{j}{\phii{\phi}{j}'}^{-1}\Big|_{\Gi{G}{i}_{y,r_{i-1}^+}} =1.$$
\end{definition}

\begin{lemma}\label{lem:equivchardata}
Let $\Sigma=(\vec{\bG},y,\vec{r},\vec{\phi})$ and $\Sigma'=(\vec{\bG},y',\vec{r},\vec{\phi'})$ be two character-data for $G$.
If $[y] = [y']$ and $\vec{\phi'}$ is a refactorization of $\vec{\phi}$, then $\groupp{\Sigma} = \groupp{\Sigma'}$ and $\chara{\Sigma} = \chara{\Sigma'}$.
\end{lemma}

\begin{proof}
The group $\groupp{\Sigma}$ depends only on the first three terms of $\Sigma$ (and in fact only on the image of $y$ in $\builred(G)$), so the equality of groups is immediate. The equality of characters follows from \cite[Lemma 4.23]{HakimMurnaghan2008}. 
\end{proof}

\section{Restricting characters}\label{sec:restrict}

In this section, we define a $\Gamma$-stable character-datum $\Sigma$ for $G$ and outline how to obtain from it a character-datum $\Sigma^\Gamma$ for $G^{[\Gamma]}$ such that the corresponding semisimple characters are naturally related via restriction, following \cite{LathamNevins2023}.

We begin by discussing our notion of $\Gamma$-stability.  Recall that $[y]\in \builred(G)$ denotes the projection of $y\in \buil(G)$.

\begin{definition}\label{def:gammastabledatum}
Let $\Sigma = (\vec{\bG},y,\vec{r},\vec{\phi})$ be a character-datum for $G$. Then $\Sigma$ is said to be $\Gamma$-stable if 
\begin{enumerate}
\item $[y]$ is $\Gamma$-fixed, and
\item for each $0\leq i \leq d$ and for each $\X{i}\in \LieZ^{i*}_{-r_i}$ that realizes $\phii{\phi}{i}$ on $\Gi{G}{i}_{y,s_i^+}$, the coset $\X{i} + \LieZ^{i*}_{-s_i}$ is $\Gamma$-stable.
\end{enumerate}
\end{definition}

In contrast, in 
\cite[Definition 3.11]{LathamNevins2023}, both $y$ and specific choices of $\X{i}$ were deemed to be fixed by $\Gamma$. Note that the $\X{i}$ considered therein were almost strongly stable in the sense of Remark~\ref{rem:genericgeneral}. The next two lemmas establish the equivalence of our two notions of $\Gamma$-stability.

\begin{lemma}\label{lem:gammafixedy}
    Let $x\in \buil(G)$. 
    The following are equivalent.
    \begin{itemize}
    \item[1)] The point $[x]\in \builred(G)$ is $\Gamma$-fixed.
    \item[2)]  There exists $\tilde{x}\in \buil(G)^\Gamma$ such that $[\tilde{x}]=[x]$.
    \end{itemize}    
\end{lemma}

\begin{proof}
Since $\buil(G)=\builred(G)
\times X_*(Z(\bG),F)\otimes_{\mathbb{Z}} \mathbb{R}$, with $\Gamma$ acting on each factor separately, (2) implies (1).  Conversely, assume that $[x]$ is $\Gamma$-fixed.  Then  the barycentre  $\tilde{x} =\frac{1}{|\Gamma|}\sum_{\gamma\in\Gamma}\gamma\cdot x$ of the orbit of $x$ under $\Gamma$ will again have the same projection $[x]\in \builred(G)$ and by construction it is $\Gamma$-fixed.
\end{proof}


\begin{lemma}\label{lem:defstabilityequiv}
Let $\bG'$ be a twisted Levi subgroup of $\bG$ and $y\in \buil(G')$.  Let $\phi$ be a $G$-generic quasicharacter of $G'$ of depth $r$. Denote by $\LieZ^{'*}$ the dual of the Lie algebra of the center of $\bG'$ and set $s=r/2$. Then the following are equivalent.
\begin{itemize}
\item[1)] For any $X\in \LieZ^{'*}_{-r}$ realizing $\phi$ on $G_{y,s^+}$, the coset $X+\LieZ^{'*}_{-s}$ is $\Gamma$-stable.
\item[2)] There exists a $\Gamma$-fixed almost strongly stable element $X\in \LieG^{'*}$ realizing $\phi$ on $G_{y,s^+}$ and satisfying $\Cent_{\bG}(X)=\bG'$.
\end{itemize}
\end{lemma}

\begin{proof}
Suppose $X'\in \LieZ^{'*}_{-r}$ realizes $\phi$ on $G'_{y,s^+}$.  By Lemma~\ref{L:etasXdependsoncoset} this is a property of the coset $X'+\LieZ^{'*}_{-s}$.  If this coset is $\Gamma$-stable, then we may set 
$$
\check{X} = \frac{1}{\vert \Gamma \vert}\sum_{\gamma\in \Gamma}\gamma\cdot 
X' \quad \in \frac{1}{\vert \Gamma \vert}\sum_{\gamma\in \Gamma}\gamma\cdot (X'+\LieZ^{'*}_{-s}) = X'+\LieZ^{'*}_{-s}
$$
where we have used that $\val(\vert \Gamma \vert)=0$ to ensure that scaling  by this element preserves depth.  Then $\check{X}$ is a $\Gamma$-fixed element realizing $\phi$ on $G'_{y,s^+}$.  Since $\phi$ is $G$-generic, so is $\check{X}$; thus  by Remark~\ref{rem:genericgeneral} it is almost strongly stable. Since $\check{X}\in \LieZ^{'*}$, its centralizer contains $\bG'$ and genericity assures us that equality holds.

Conversely, if $X\in \LieG^{'*}$ satisfies $\Cent_{\bG}(X)=\bG'$ then since our tameness hypothesis implies we may $G'$-equivariantly identify $\LieG'$  with $\LieG^{'*}$, 
we have  $X\in (\LieG^{'*})^{\bG'}\cong \LieZ^{'*}$.  Thus if $X$ realizes $\phi$ on  $G'_{y,s^+}$, then it lies in $\LieZ^{'*}_{-r}$ and if it is $\Gamma$-fixed then both $\bG'$ and the coset $X+\LieZ^{'*}_{-s}$ are $\Gamma$-stable.  
\end{proof}

With this preparation, we can now show \cite[Proposition 3.12 and Theorem 4.4]{LathamNevins2023} hold with our notion of $\Gamma$-stable character-data.

\begin{theorem}\label{th:restriction}
Let $\Sigma = (\vec{\bG},y,\vec{r},\vec{\phi})$ be a character-datum for $G$ which is $\Gamma$-stable in the sense of Definition \ref{def:gammastabledatum}. Define $\Sigma^\Gamma = (\vec{\bG}^{[\Gamma]},y,\vec{r},\vec{\phi})$, where
\begin{enumerate}
\item[(i)] $\vec{\bG}^{[\Gamma]}$ is the twisted Levi sequence in  $\bG^{[\Gamma]}$ consisting of the groups $\Gi{\bG}{i,[\Gamma]} = (\Gi{\bG}{i})^{[\Gamma]}$;
\item[(ii)] each quasicharacter $\phii{\phi}{i}$ is viewed as a quasicharacter of $\Gi{\bG}{i,[\Gamma]}$ by restriction. 
\end{enumerate}
For each $i$ such that there exists $k>0$ for which $\Gi{\bG}{i,[\Gamma]} = \Gi{\bG}{i+k,[\Gamma]}$, replace the tuple of subsequences $((\Gi{\bG}{i,[\Gamma]},\dots,\Gi{\bG}{i+k,[\Gamma]}),(r_i,\dots,r_{i+k}),(\phii{\phi}{i},\dots,\phii{\phi}{i+k}))$ with the tuple $\displaystyle (\Gi{\bG}{i,[\Gamma]}, r_{i+k}, \prod_{j=0}^{k}\phii{\phi}{i+j})$. Then the resulting tuple $\Sigma^\Gamma$ is a character-datum for $G^{[\Gamma]}$. Furthermore, $\groupp{\Sigma}^\Gamma = \groupp{\Sigma^\Gamma}$ and $\chara{\Sigma^\Gamma} = \chara{\Sigma}|_{\groupp{\Sigma}^\Gamma}$.

\end{theorem}

In an abuse of language, we will say that $\Sigma^\Gamma$ is the $\Gamma$-fixed point of $\Sigma$.
We illustrate Theorem~\ref{th:restriction} in Figure~\ref{fig:restrictionpicture}.

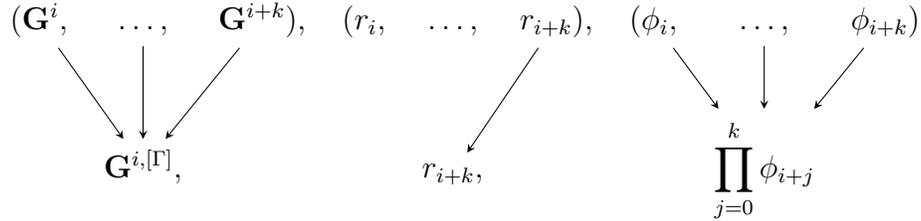
\begin{figure}[!htbp]
\begin{tikzpicture}
  \matrix (m) [matrix of math nodes,row sep=2em,column sep=0.5em,minimum width=0.1em]
  {
    (\Gi{\bG}{i}, &\dots, &\Gi{\bG}{i+k}), & (r_i, &\dots, & r_{i+k}),& (\phii{\phi}{i}, &\dots, & \phii{\phi}{i+k}) \\
    &   \Gi{\bG}{i,[\Gamma]}, & & &r_{i+k}, & & & \displaystyle\prod_{j=0}^{k}\phii{\phi}{i+j}\\};
  \path[-stealth]
  (m-1-1) edge node [left] {} (m-2-2)
  (m-1-2) edge node [right] {} (m-2-2)
  (m-1-3) edge node [right] {} (m-2-2)
  (m-1-6) edge node [left] {} (m-2-5)
  (m-1-7) edge node [left] {} (m-2-8)
  (m-1-8) edge node [right] {} (m-2-8)
  (m-1-9) edge node [right] {} (m-2-8);
\end{tikzpicture}
\caption{\label{fig:restrictionpicture}Illustration of how part of the character-datum collapses when taking the $\Gamma$-fixed point as per Theorem \ref{th:restriction}. Here, $\Gi{\bG}{i,[\Gamma]} = \Gi{\bG}{i+1,[\Gamma]} = \dots = \Gi{\bG}{i+k,[\Gamma]}$.}
\end{figure}

\begin{proof}[Proof of Theorem \ref{th:restriction}]
Given a $\Gamma$-stable character-datum $\Sigma = (\vec{\bG},y,\vec{r},\vec{\phi})$, we may assume without loss of generality that $y$ is $\Gamma$-fixed as a consequence of Lemmas \ref{lem:equivchardata} and \ref{lem:gammafixedy}.
Thus we have $y \in \buil(G)^\Gamma = \buil(G^{[\Gamma]})$. By Lemma \ref{lem:defstabilityequiv}, there exists a $\Gamma$-fixed $\Gi{G}{i+1}$-generic element $\X{i} \in \LieZ^{i*}_{-r_i}$ that realizes $\phii{\phi}{i}$ on $\Gi{G}{i}_{y,s_i^+}$. 
Let $\bT$ be any maximal torus of $\bG^{i}$.  By the independence of genericity from the choice of $\bT$ (Remark~\ref{rem:genericgeneral}), we have $\Cent_{\Gi{\bG}{i+1}}(\X{i}) = \langle \bT, \bU_\alpha : \X{i}(H_\alpha) = 0 \rangle = \bG^i$. It follows that $\Sigma' = (\vec{\bG},y,\vec{r},\vec{X},\vec{\phi})$, where $\vec{X} = (\X{0},\dots,\X{d})$, is a $\Gamma$-stable truncated datum in the sense of \cite[Definitions 2.1 and 3.11]{LathamNevins2023}. The result thus follows from \cite[Proposition 3.12 and Theorem 4.4]{LathamNevins2023}.
\end{proof}

\section{Lifting characters via a lift of data}\label{sec:lift}

The main result of this section is the explicit construction of a lift for a character-datum $\Delta$ of $G^{[\Gamma]}$, as per the following theorem.

\begin{theorem}\label{th:lift}
Let $\Delta$ be a character-datum for $G^{[\Gamma]}$. Then there exists a character-datum $\Sigma$ for $G$ such that 
\begin{enumerate}
\item[1)] $\Sigma$ is $\Gamma$-stable in the sense of Definition \ref{def:gammastabledatum}, and \item[2)]$\Sigma^\Gamma$, as defined in Theorem \ref{th:restriction}, is a refactorization of $\Delta$ in the sense of Definition \ref{def:refactorization}. 
\end{enumerate} 
In particular, $\groupp{\Delta} = \groupp{\Sigma^\Gamma}$ and $\chara\Delta = \chara\Sigma|_{\groupp{\Delta}}$.
\end{theorem}

Our proof will be constructive.  To motivate it, let us give a quick example to illustrate how a direct approach would fail.
Set $\bH=\bG^{[\Gamma]}$ and suppose $\Delta = (\vec{\bH},y,\vec{t},\vec{\xi})$ is a character-datum for $H$.  
As a first attempt, one might simply choose to fix a Levi sequence of $\bG$ whose fixed points give $\bH$. 

\begin{example}\label{Eg:weird}
Let $G=\GL(4,F)$ and $\Gamma = \{1,\gamma\}$ such that  $H=G^{[\Gamma]}=\GL(2,F)\times \GL(2,F)$.    Consider the \singlequasicharacter\ character-datum
$$
\Delta = (\Gi{\bH}{0}=\bH, y, t, \phii{\xi}{0}=\chi_0\circ \det\otimes \chi_1\circ \det)
$$
where $\chi_0, \chi_1$ are two characters of $F^\times$  of depth $t >0$. 

There are two $\Gamma$-stable Levi subgroups of $\bG$ whose $\Gamma$-fixed points are $\bH$: $\bG$ and $\bH$.  With the choice $\bG^0=\bG$, there exists a lift of $\phii{\xi}{0}$ 
to a quasicharacter of $\Gi{\bG}{0}$ if and only if $\chi_0=\chi_1$. 
With the choice of $\Gi{\bG}{0}=\bH$, we may always view $\phii{\xi}{0}$
as a quasicharacter of $\bG^0$ --- but $\phii{\xi}{0}$
will fail to be $G$-generic whenever $\chi_0|_{F^\times_{t}}=\chi_1|_{F^\times_{t}}$ (Remark \ref{rem:genexplained}).  Thus in general, there does not exist a direct lift of $\Delta$ to a \singlequasicharacter\ character-datum for $G$.
\end{example}

This simple example illustrates  that the choice of twisted Levi subgroups of $\bG$ appearing in the lift of a character-datum for $H$ will depend in subtle ways on the quasicharacters appearing in $\Delta$. In the following subsection, we establish the core of this dependence.

\subsection{Lifting a single quasicharacter}\label{sec:liftchar}

The main result of this section (Theorem \ref{T:liftone}) produces a lift for a single quasicharacter of a character-datum for $G^{[\Gamma]}$. The proof of this theorem relies heavily on Theorem \ref{T:charexists}, as well as a strategic construction for producing twisted Levi subgroups (Proposition \ref{P:indstep}). This construction follows techniques from \cite{Fintzen2021} and makes use of the broader notion of genericity discussed in Remark \ref{rem:genericgeneral}.

\begin{proposition} \label{P:indstep}
Let $\bH'$ be a twisted Levi subgroup of $\bH = \bG^{[\Gamma]}$ and let $\bZ' = Z(\bH')^\circ$.  Suppose that $\bM$ is a $\Gamma$-stable twisted Levi subgroup of $\bG$ containing $\bH'$ and that $\xi$ is a quasicharacter of $H'$ of depth $t>0$. 
Then there exists an $M$-generic element $X^\sharp\in \LieZ^{'*}_{-t}$ that realizes $\xi$ on $H'_{x,t}$, for any $x\in \buil(H')$.  Moreover, if we suppose further that either
\begin{enumerate}[(a)]
\item $\bM=\bG$ and $\xi$ is $H$-generic, or 
\item $\bM^{[\Gamma]}=\bH'$,
\end{enumerate}
then  $\bM' = \Cent_{\bM}(X^\sharp)$ is a $\Gamma$-invariant twisted Levi subgroup satisfying ${\bM'}^{[\Gamma]}=\bH'$.
\end{proposition}

\begin{proof}
Choose a maximal $F$-torus $\bS$ of $\bH$ contained in $\bH'$ and let $\bT=\Cent_\bG(\bS)$ be the corresponding maximal torus of $\bG$ \cite[Lemma 3.1]{LathamNevins2023}.  Writing $\mathfrak{s}$ and $\mathfrak{t}$ for the $F$-points of  their Lie algebras, we have a unique $\Gamma$-equivariant embedding $\mathfrak{s}^*\subset \mathfrak{t}^*$.  We also embed ${\mathfrak{z}'}^*\subset \mathfrak{s}^*$ as the functionals that are trivial on $\mathfrak{s}\cap \LieH'_{der}$.  The embeddings ${\mathfrak{z}'}^*\subset \mathfrak{s}^*\subset \mathfrak{t}^*$ are depth-preserving.

Let $\bM$ be a $\Gamma$-stable twisted Levi subgroup of $\bG$ containing $\bH'$. 
Since $\bM$ is $\Gamma$-stable and contains $\bS$, it contains $\bT$.  Note that  the embedding of ${\mathfrak{z}'}^*$ into $\mathfrak{m}^*$ is independent of the choice of $\bS$ and it is with respect to this embedding that an element of $\mathfrak{z}^{'*}$ can be considered $M$-generic. 

Let $E$ denote a (tame) Galois extension over which $\bT$ splits. Then, for each $\alpha \in \Phi(\bG,\bT)$ we may consider $H_\alpha = d\alpha^\vee(1)$ as an element of $\mathfrak{t}(E) := \Lie(\bT)(E)$. 

Now let $\xi$ be a quasicharacter of $H'$ of depth $t>0$.  Let $x\in \buil(H')$.  
By Remark \ref{rem:realizedcenter},  $\xi$ is realized by an element $\X{}\in {\mathfrak{z}'}^*$ of depth $-t$ on $H'_{x,t/2^+}$. Furthermore, $\X{}$ realizes $\xi$ on $H'_{x,t}$, as well as $H'_{y,t}$ for all $y\in \buil(H')$.

Let $\Phi_M=\Phi(\bM,\bT)$ and set
$$
\Phi'=\{\alpha \in \Phi_M\mid \X{}(H_\alpha)=0 \text{ or } \val(\X{}(H_\alpha))>-t\}.
$$ 
The set $\Phi'$ is invariant under both $\Gal(E/F)$ and $\Gamma$ since $\X{}$ is. 
Since $\X{}$ has depth $-t$ it follows that for all $\alpha \in \Phi_M \smallsetminus \Phi'$ (a set that may be empty when $\bH'=\bH$) we have $\val(\X{}(H_\alpha))=-t$.  Following an argument that was used in \cite[Proposition 3.12]{Fintzen2021} and \cite[Theorem 3.3]{Fintzen2021good}, we define
$$
\LieT(E)^*_\sharp=\{X^\sharp\in \mathfrak{t}(E)^* \mid X^\sharp(H_\alpha)=0 \;\forall \alpha \in \Phi'\}
$$
and 
$$
\LieT(E)^*_\flat=\text{span}\{d\alpha \mid \alpha \in \Phi'\}.
$$
These subspaces of $\mathfrak{t}(E)^*$ are each $\Gamma$-invariant and defined over $F$.  Under our identification of $\mathfrak{t}^*$ and $\mathfrak{t}$, $\LieT(E)^*_\flat$ corresponds to $\text{span}\{H_\alpha \mid \alpha \in \Phi'\}$, whence $\mathfrak{t}(E)^*=\LieT(E)^*_\sharp\oplus \LieT(E)^*_\flat$.  Thus there exist unique $X^\sharp\in \LieT(E)^*_\sharp$ and $X^\flat\in \LieT(E)^*_\flat$ such that $\X{}=X^\sharp+X^\flat$. By construction,  $X^\flat$ has depth strictly larger than $-t$.  

Since $\X{}$ is fixed by both $\Gal(E/F)$ and $\Gamma$, the invariance of these two spaces individually implies  $X^\sharp, X^\flat \in \mathfrak{s}^*$. Given that $X^\sharp \in \mathfrak{s}^*$, its centralizer $\bM' = \Cent_{\bM}(X^\sharp)$ is a twisted Levi subgroup of $\bM$, and therefore of $\bG$, that is $\Gamma$-stable. We claim its root system is $\Phi'$ and that $X^\sharp$ is $M$-generic of depth $-t$.  Since $X^\sharp \in \X{} + \mathfrak{z}^{'*}_{(-t)^+}$, this would complete the proof of the first statement.

Namely, if $\alpha \in \Phi'$ then $X^\sharp(H_\alpha)=0$ by definition of $\LieT(E)^*_\sharp$.  If $\alpha \in \Phi_M\setminus \Phi'$ then since $X^\flat\in {\mathfrak{z}'}^*_{(-t)^+}$ we have
$$
\val(X^\sharp(H_\alpha)) = \val(\X{}(H_\alpha)-X^\flat(H_\alpha))=\val(\X{}(H_\alpha))=-t,
$$
as required.  

It remains to show that $\bM'^{[\Gamma]}=\bH'$ if either (a) or (b) hold. We deduce from \cite[Lemma 3.3]{LathamNevins2023} that, since $\bM'$ is a twisted Levi subgroup of $\bM$ such that $\Phi(\bM',\bT)=\Phi'$, this is equivalent to showing that 
$$
\{a \in \Phi(\bH,\bS) : a = \alpha|_{\bS} \text{\ for some\ } \alpha \in \Phi'\}=\Phi(\bH',\bS).
$$

In case (a), we assume $\bM=\bG$ and $\xi$ is $H$-generic.  Thus, by Lemma~\ref{lem:realizedindependence}, 
any element $\X{}$ that realizes $\xi$ is an $H$-generic element of ${\mathfrak{z}'}^*$ of depth $-t$.  This implies $\X{}(H_a)=0$ for all $a\in \Phi(\bH',\bS)$ and that $\val(\X{}(H_a))=-t$ for all $a\in \Phi(\bH,\bS)\smallsetminus \Phi(\bH',\bS)$.  Since $\X{}$ is $\Gamma$-fixed, $\X{}(H_\alpha)=\X{}(H_a)$ for every root $\alpha \in \Phi(\bG,\bT)$ that restricts to a root of $\bH$.  By definition, $\alpha\in \Phi'$ if and only if either $\X{}(H_\alpha)=0$ or $\val(\X{}(H_\alpha))>-t$; thus if $\alpha\in \Phi'$ restricts to a root of $\bH$ then we must have $\X{}(H_\alpha)=\X{}(H_a)=0$.  Therefore, every root $\alpha \in \Phi'$ that restricts to a root of $\bH$ in fact must restrict to a root of $\bH'$; and furthermore, every root of $\Phi_{\bM}\smallsetminus \Phi'$ that restricts to a root of $\bH$ must restrict to an element of $\Phi(\bH,\bS)\smallsetminus \Phi(\bH',\bS)$.  

In case (b), we assume $\bM^{[\Gamma]}=\bH'$ but not that $\xi$ (and $\X{}$) are $H$-generic.  This implies
$$
\{a \in \Phi(\bH,\bS) : a = \alpha|_{\bS} \text{\ for some\ } \alpha \in \Phi_{\bM}\}=\Phi(\bH',\bS).
$$
In particular, every root $\alpha\in \Phi'\subset \Phi_M$ that restricts to a root of $\bH$ satisfies $\alpha|_{\bS} \in \Phi(\bH',\bS)$.  Conversely if $\alpha\in \Phi_M$ is a restrictable root then its  restriction is an element $a \in \Phi(\bH',\bS)$.  Since $\X{}\in \mathfrak{z}^{'*}$ we have $\X{}(H_\alpha)=\X{}(H_a)=0$, whence $\alpha \in \Phi'$. 
\end{proof}

Let us illustrate the construction in this proof with an example.

\begin{example}\label{Eg:P:indstep}
Consider $G=\GL(4,F)$, $\Gamma$, and $H=G^{[\Gamma]}=\GL(2,F)\times \GL(2,F)$ as in Example~\ref{Eg:weird}.  Choose a quadratic extension field $E$ of $F$ and a subgroup $\bT\cong \Res_{E/F}\bG_m$ of $\GL(2)$ and set $\Gi{\bH}{0}=\bT\times \bT$.   For $y\in \buil(\Gi{\bH}{0})$ consider the \singlequasicharacter\ character-datum
$$
\Delta = (\bH^0\subsetneq \bH, y, t, \xi:=\eta\otimes\eta'),
$$  
for two $\GL(2,F)$-generic quasicharacters $\eta$, $\eta'$ of $T$ of depth $t$.
We apply Proposition~\ref{P:indstep} with  $\bH'=\bH^0$, $\bM=\bG$ and the $H$-generic character $\xi$. 
Since $H'$ is abelian, $Z'=H'$.  
An element $X\in (\LieH')^*_{-t}$ realizing $\xi|_{H'_{y, t}}$ is of the form $X=(X_1,X_2)$ with $X_i\in \LieT^*_{-t}$.

Suppose first $\eta|_{T_{t}}=\eta'|_{T_{t}}$, so that
$X_1-X_2\in \LieT^*_{(-t)^+}$.  The proof of Proposition~\ref{P:indstep} factors $X$ as $X=X^\sharp+X^\flat$ such that $$X^\sharp = (Y,Y)$$ with $Y-X_i\in \LieT^*_{(-t)^+}$ and $X^\flat = (X_1-Y,X_2-Y) \in \LieT^*_{(-t)^+}$.  We obtain $$\bM' = \Cent_{\bG}(X^\sharp)\cong\Res_{E/F}\GL(2),$$ which is a $\Gamma$-invariant twisted Levi subgroup of $\bM$ satisfying  ${\bM'}^{[\Gamma]}=\bH'$ and relative to which $X^\sharp$ is $M$-generic element.  A similar situation occurs if $\eta|_{T_{t}}=(\eta')^{-1}|_{T_{t}}$, with $X^\sharp=(Y,\overline{Y})$ where $\overline{Y}$ represents the Galois conjugate of $Y$ as an element of $E$.    
For all other $\eta,\eta'$, the root system $\Phi'$ arising in the proof is empty and thus $\LieT(E)_\flat^*=\{0\}$, yielding $\bM'=\bT\times \bT$ and the $G$-generic element $X^\sharp=(X_1,X_2)$.
\end{example}

Our next result is to show how to use the construction  of Proposition~\ref{P:indstep} iteratively to produce a lift to $G$ of any \singlequasicharacter\  character-datum for $H$.

\begin{theorem}\label{T:liftone}
Let $\bH'$ be a twisted Levi subgroup of $\bH=\bG^{[\Gamma]}$ defined over $F$ and let $x\in\buil(H')$.  Suppose $\xi$ is an $H$-generic quasicharacter of $H'$ of depth $t>0$.  Set $\Delta = (\bH'\subseteq \bH, x, t, \xi)$. Then for any $0\leq s <t$, there exists a $\Gamma$-stable character-datum $\Sigma_0=(\vec{\bG},x,\vec{r},\vec{\phi})$ for $G$ 
and a quasicharacter $\varphi$ of $H'$ of depth at most $s$ such that $\Sigma_0^\Gamma = (\bH'\subseteq \bH, x, t, \varphi \xi)$.  More precisely, for some $n\geq 0$ we have:
\begin{enumerate}
\item[1)] $s< r_0 < r_1 < \cdots < r_n = t$;
\item[2)] $\Gi{\bG}{0}\subsetneq \cdots \subsetneq \Gi{\bG}{n}\subseteq \Gi{\bG}{n+1} = \bG$ is a sequence of twisted Levi subgroups of $\bG$ such that for all $0\leq i \leq n$ we have $(\Gi{\bG}{i})^{[\Gamma]}=\bH'$;
\item[3)] $\phii{\phi}{i}$ is $\Gi{G}{i+1}$-generic quasicharacter of $\Gi{G}{i}$ of depth $r_i$ for all $0\leq i\leq n$; and
\item[4)] $\displaystyle\varphi\xi = \prod_{i=0}^n \phii{\phi}{i}|_{H'}.$
\end{enumerate}
\end{theorem}

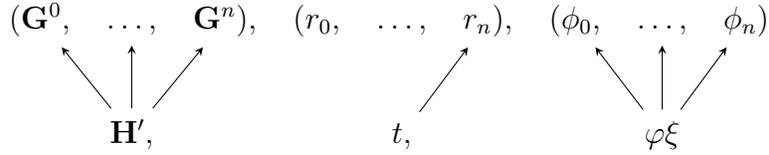
\begin{figure}[!htbp]
\begin{tikzpicture}
  \matrix (m) [matrix of math nodes,row sep=2em,column sep=0.5em,minimum width=0.1em]
  {
    (\Gi{\bG}{0}, &\dots, &\Gi{\bG}{n}), & (r_0, &\dots, & r_n),& (\phii{\phi}{0}, &\dots, & \phii{\phi}{n}) \\
    &   \bH', & & &t, & & & \varphi\xi\\};
  \path[-stealth]
  (m-2-2) edge node [left] {} (m-1-1)
  (m-2-2) edge node [right] {} (m-1-2)
  (m-2-2) edge node [right] {} (m-1-3)
  (m-2-5) edge node [left] {} (m-1-6)
  (m-2-8) edge node [left] {} (m-1-7)
  (m-2-8) edge node [right] {} (m-1-8)
  (m-2-8) edge node [right] {} (m-1-9);
\end{tikzpicture}
\caption{Lift (top) produced by Theorem \ref{T:liftone}, given the $H$-generic quasicharacter $\xi$ of $\bH'$ of depth $t=r_n$ (bottom). The $\Gamma$-fixed point of the lift is $(\bH',t,\varphi\xi)$.}
\end{figure}

\begin{proof}[Proof of Theorem \ref{T:liftone}]
Set $\bZ'= Z(\bH')^\circ$ and suppose $\xi$ is an $H$-generic quasicharacter of $H'$ of depth $t > 0$. By case (a) of Proposition~\ref{P:indstep} (applied with $\bM=\bG$), there exists a $G$-generic element $X^\sharp\in\LieZ^{'*}_{-t}$ that realizes $\xi|_{H'_{x,t}}$. 
Furthermore, $\bM_1 = \Cent_\bG(X^\sharp)$ is a $\Gamma$-invariant Levi subgroup of $\bG$ satisfying $\bM_1^{[\Gamma]} = \bH'$. By Theorem \ref{T:charexists}, we may choose 
a quasicharacter $\phii{\phi}{1}'$ of $M_1$ of depth $t$ which is realized by the $\Gamma$-fixed element $X^\sharp$ on $(M_1)_{x,t/2^+}$. 

For any such choice of $\phii{\phi}{1}'$, 
$\xi$ and $\phii{\phi}{1}'$ restrict to the same quasicharacter of $H'_{x,t}$. 
It follows that the quasicharacter $\tau_1 = \xi(\phii{\phi}{1}')^{-1}|_{H'}$ has depth $t_1 < t$. 

If $t_1\leq s$, we are done. Indeed, we set $n=0, r_0=t, \Gi{\bG}{0} = \bM_1, \phii{\phi}{0} = \phii{\phi}{1}'$ and $\varphi = \tau_1^{-1}$. By construction, $\Gi{\bG}{0}$ is a twisted Levi subgroup of $\bG$, $\phii{\phi}{0}$ if $G$-generic of depth $r_0$, $\varphi$ is a quasicharacter of $H'$ of depth $t_1\leq s < r_0$ and $\varphi\xi = \phii{\phi}{0}|_{H'}$. Furthermore, the character-datum $(\Gi{\bG}{0}\subseteq \bG, x, r_0, \phii{\phi}{0})$ is $\Gamma$-stable in the sense of Definition \ref{def:gammastabledatum}, with $\Gamma$-fixed point equal to $(\bH'\subseteq \bH, x, t, \varphi \xi)$.

If $t_1 > s$, then set $t_0=t$ and $X_0^\sharp=X^\sharp$.  We proceed in two phases.  The first is to inductively construct tuples $(\bM_{i}, \phii{\phi}{i}', X_{i-1}^\sharp, \tau_i, t_i)$ for $i> 1$ such that 
\begin{itemize}
\item $\bM_i^{[\Gamma]}=\bH'$, 
\item $\phii{\phi}{i}'$ is a $M_{i-1}$-generic character of $M_i$ of depth $t_{i-1}$ realized on $(M_i)_{x,(t_{i-1}/2)+}$ by the ($\Gamma$-fixed) element $X_{i-1}^\sharp\in \LieZ^{'*}_{-(t_{i-1})}$ whose centralizer in $\bM_{i-1}$ is $\bM_i$, and
\item $\tau_i$ is a character of $H'$ of depth $t_i<t_{i-1}$.
\end{itemize}
We stop when $t_i\leq s$.

Suppose the tuple for $i$ has been constructed.  Applying case (b) of Proposition \ref{P:indstep} to $\bM=\bM_i$, $\xi=\tau_i$ of depth $t=t_i$ yields that we may choose a $M_i$-generic element $X_i^\sharp\in \LieZ^{'*}_{-t_i}$ that realizes $\tau_i|_{H'_{x,t_i}}$, as well as a $\Gamma$-invariant twisted Levi subgroup $\bM_{i+1} = \Cent_{\bM_i}(X_i^\sharp)$ of $\bM_i$  
satisfying $\bM_{i+1}^{[\Gamma]} = \bH'$. As done above, Theorem \ref{T:charexists} allows us to choose a quasicharacter $\phii{\phi}{i+1}'$ of $M_{i+1}$ of depth $t_{i}$ which is realized by $X_i^\sharp$ on $(M_{i+1})_{x,t_i/2^+}$. By construction, $X_i^\sharp$  is $\Gamma$-fixed and $\phii{\phi}{i+1}'$ is $M_i$-generic of depth $t_i$.  Furthermore, 
$\phii{\phi}{i+1}'$ and $\tau_{i}$ have the same restriction to $H'_{x,t_i}$. 
Setting $\tau_{i+1} = \tau_i(\phii{\phi}{i+1}')^{-1}|_{H'}$, we have that $\tau_{i+1}$ is a quasicharacter of $H'$ of depth $t_{i+1} < t_i$.  

Note that this process is guaranteed to terminate in a finite number of steps since $t_{i+1} < t_i$ for all $i\geq 0$, and each $t_i$ corresponds to the depth of a character, which is a discrete set.  Let $m=i+1$ denote the first index for which $t_{i+1}\leq s$.

We have produced a nested sequence of (not necessarily distinct) twisted Levi subgroups
$$
\bM_m \subseteq \bM_{m-1}\subseteq \cdots \subseteq \bM_{2}\subseteq \bM_1.
$$
The second phase  is to gather the tuples above into a character-datum.  

Let $n+1$ be the number of distinct twisted Levi subgroups of this sequence. Rename and reindex the resulting sequence as
$$
\Gi{\bG}{0}\subsetneq \Gi{\bG}{1} \subsetneq \cdots \subsetneq \Gi{\bG}{n} \subseteq \Gi{\bG}{n+1} =\bG.
$$
For each $0\leq k \leq n$ let $j_k$ be the largest index such that $\bM_{j_k}=\Gi{\bG}{k}$.  Then $j_0=m>j_1>\cdots>j_n\geq 1$ and we have
$$
\Gi{\bG}{k}=\bM_{j_k}=\cdots=\bM_{j_{k+1}+1}.
$$
By construction, each $\Gi{\bG}{k}$ satisfies $(\Gi{\bG}{k})^{[\Gamma]}=\bH'$.
Set
$$
r_k := t_{j_{k+1}+1} \quad \text{and} \quad \phii{\phi}{k} = \prod^{j_{k+1}+1}_{i=j_k}\phii{\phi}{i}'.
$$
Since the quasicharacters $\phii{\phi}{i}'$ with 
$j_{k+1}+1<i\leq j_k$ 
have depths strictly less than $r_k$,   the quasicharacter $\phii{\phi}{k}$ of $G^k$ has depth $r_k$. Moreover, $\phii{\phi}{k}|_{G^k_{x,r_k}}$ coincides with $\phii{\phi}{j_{k+1}+1}'|_{(M_{j_{k+1}+1})_{x,t_{j_{k+1}+1}}}$, which is $M_{j_{k+1}}=\Gi{G}{k+1}$-generic;  thus $\phii{\phi}{k}$ is $\Gi{G}{k+1}$-generic of depth $r_k$. Furthermore, $\phii{\phi}{k}$ is realized by the $\Gamma$-fixed element $\displaystyle \sum_{i=j_k}^{j_{k+1}+1}X_{i-1}^\sharp$ on $\Gi{G}{k}_{x,r_k/2+}$.  
Thus,  
the character-datum $(\vec{\bG},x,\vec{r},\vec{\phi})$ is $\Gamma$-stable by Lemma \ref{lem:defstabilityequiv}.

By our choice of $m$ we have $s < t_{m-1}\leq r_0 < r_1 < \cdots < r_n = t$. Set $\varphi=\tau_m^{-1}$.  This is a quasicharacter of $\bH'$ of depth $t_{m}\leq s$ and we have 
$$
\varphi\xi = \prod_{i=1}^{m}\phii{\phi}{i}'|_{H'} = \prod_{k=0}^n \phii{\phi}{k}|_{H'},
$$
by construction.
\end{proof}

Let us continue our running example to illustrate the application of Theorem~\ref{T:liftone}.  

\begin{example} \label{Eg:T:liftone}
Let $\Delta$ be as in Example~\ref{Eg:P:indstep} and suppose we are in the first (most interesting) case that yields $X^\sharp=(Y,Y)$, $\bM_1=\Res_{E/F}\GL(2)$.  We set $s=0$ as our target depth.

Following the proof of Theorem~\ref{T:liftone}, we must choose a quasicharacter $\phi_1'$ of $M_1$ realized by $X^\sharp$ on $(M_1)_{x,t/2^+}$.  
Such a character is of the form $\zeta_1\circ \det_E$ where $\det_E:\GL(2,E)\to E^\times$ is the determinant map, and $\zeta_1$ is a character of $E^\times$ of depth $t$.  This leads to $\tau_1=\eta(\zeta_1^{-1})\otimes \eta'(\zeta_1^{-1})$ of strictly lower depth $t_1$; say it is realized by $(Y_1,Y_1')$, which need not be $H$-generic.

We now repeat the application of Proposition~\ref{P:indstep}, relative to $\bM=\bM_1$.  If $Y_1-Y_1'\in \LieT^*_{x,(-t_1)+}$ then in the proof of Proposition~\ref{P:indstep} we have $\Phi'=\Phi_M$, which yields $\bM_2=\bM_1$; thus in the proof of Theorem~\ref{T:liftone} we repeat the process of the preceding paragraph.

So suppose that after $i$ iterations we finally have $X_{i-1}^\sharp=(Y_i,Y_i')$ such that $Y_i-Y_i'\notin \LieT^*_{x,(-t_i)+}$.  Then $\Phi'$ is empty, yielding $X_i^\sharp = X_{i-1}^\sharp$, $\bM_{i+1}=\bT\times \bT$ and $X_i^\sharp$ is $M_1$-generic.  A quasicharacter $\phi_{i+1}'$ of $M_{i+1}=T\times T$ realized by $X_i^\sharp$ 
has depth 
$t':=t_{i+1}<t_i$.  Successive iterations will produce characters of the same group $M':=M_{i+1}$ (which need not be $M_1$-generic) of ever-decreasing depth until $t_m=0$.  Then $\Sigma_0 = (\Gi{\bG}{0}=\bM' \subsetneq \Gi{\bG}{1}=\bM_1 \subsetneq \Gi{\bG}{2}=\bG, x, (t',t), (\phi_0, \zeta\circ \det))$ where $\phi_0$ and $\zeta$ are each products of the quasicharacters constructed along the way.

Then in particular the character $\xi$ 
has been factored, up to the depth-zero twist by $\tau_m$, 
as a product of an $M_1$-generic character $\phi_0\otimes \phi_0'$ of $\bM_2=\bH'$ and the restriction of  a $G$-generic character $\zeta\circ \det_E$ of $\bM_1$. 
\end{example}

We extract an observation from the proof of Theorem~\ref{T:liftone}.

Given a \singlequasicharacter\ character-datum $\Delta = (\bH'\subseteq \bH,x,t,\xi)$ for $H$, the
inductive proof of  Theorem~\ref{T:liftone} relies at each step $i\geq 0$ on choosing two elements (notation as in the proof):
\begin{itemize}
\item an $M_{i}$-generic element $X_{i}^\sharp\in\LieZ^{'*}_{-t_{i}}$ that realizes a character $\tau_i|_{H'_{x,t_i}}$ (whose existence is proven via a sample construction in Proposition~\ref{P:indstep});  and
\item a quasicharacter $\phii{\phi}{i+1}'$ of $M_{i+1}$ of depth $t_{i}$ which is realized by $X_{i}^\sharp$ on $(M_{i+1})_{x,t_{i}/2^+}$ (whose existence is guaranteed by Theorem~\ref{T:charexists}).
\end{itemize}
Given a choice of $s$ and of these inputs, the output of Theorem~\ref{T:liftone} is a character-datum $\Sigma_0$ for $G$, together with a quasicharacter $\varphi$ of $H'$ of depth at most $s$.
Let $\mathrm{Lift}(\Delta)$ denote the set of couples $(\Sigma_0,\varphi)$ obtained by varying $0\leq s < t$ and $X_i^\sharp$ and $\phii{\phi}{i+1}'$ over all valid choices.

\begin{proposition}\label{prop:resthenliftsingle}
Let $\Sigma$ be a $\Gamma$-stable character-datum for $G$ such that $\Sigma^\Gamma$ is a single-quasicharacter character-datum for $H$. Then $(\Sigma,\triv)\in \mathrm{Lift}(\Sigma^\Gamma)$.
\end{proposition}

\begin{proof}
Let $n+1$ be the length of the character-datum $\Sigma = (\vec{\bG},y,\vec{r},\vec{\phi})$. Given our hypothesis, there exists a twisted Levi subgroup $\bH'\subseteq \bH$ such that $\Gi{\bG}{j,[\Gamma]}=\bH'$ for all $0\leq j\leq n$. By Lemmas~\ref{lem:equivchardata} and \ref{lem:gammafixedy} we may assume $y \in \buil(H')$.
Since $\Sigma$ is $\Gamma$-stable, for all $0\leq j\leq n$, there exists a $\Gamma$-fixed element $\X{j}\in\LieZ^{j*}_{-r_j}$ that is $G^{j+1}$-generic of depth $-r_j$, realizes $\phii{\phi}{j}$ on $\Gi{G}{j}_{y,s_j^+}$, and satisfies $\Cent_\bG(\X{j}) = \Gi{\bG}{j}$ (Lemma~\ref{lem:defstabilityequiv}). Then the $\Gamma$-fixed point of $\Sigma$ is
$$\Sigma^\Gamma = (\bH'\subseteq \bH, y, r_n, \xi),$$
where $\displaystyle \xi = \prod_{j=0}^n\phii{\phi}{j}|_{H'}$ is a character of depth $r_n$ realized by $\displaystyle \sum_{j=0}^n \X{j}$.  

We now go through the steps of the proof of Theorem~\ref{T:liftone} with $s=0$, making valid choices of $(X_i^\sharp,\phi_{i+1}')$ at each stage that result in the couple $(\Sigma,\triv)$. 

At the $i=1$ step, note that since $\xi$ and $\phi_n$ have the same restriction to $H'_{y,r_n}$, we may choose $X_0^\sharp=X_n$.  Then $\bM_1 = \Cent_\bG(\X{n})=\Gi{\bG}{n}$ and we may choose $\phi_{1}'=\phii{\phi}{n}$ since this is a quasicharacter of $M_1$ of depth $r_n$ which is realized by the element $X_0^\sharp\in\LieZ^{'*}_{-r_n}$ on $(M_1)_{y,r_n/2^+}$.  Then $\tau_1 = \xi(\phii{\phi}{1}')^{-1}|_{H'} = \prod_{j=0}^{n-1}\phii{\phi}{j}|_{H'}$, where this is a quasicharacter of depth $r_{n-1} < r_n$. We have thus produced the tuple $$\left( \bM_1 = \bG^n, \,\phii{\phi}{1}' = \phii{\phi}{n}, \,\X{0}^\sharp = \X{n}, \,\tau_1 = \prod_{j=0}^{n-1}\phii{\phi}{j}, \, t_1 = r_{n-1} \right).$$
Suppose that for some step $1\leq i\leq n$ we have produced the tuple
$$
\left( \bM_i = \bG^{n-i+1}, \,\phii{\phi}{i}' = \phii{\phi}{n-i+1}, \,\X{i-1}^\sharp = \X{n-i+1}, \,\tau_i = \prod_{j=0}^{n-i}\phii{\phi}{j}, \, t_i = r_{n-i} \right).
$$ 
Since $\tau_i$ and $\phi_{n-i}$ agree on $H'_{y,t_i}$, which is realized by $X_{n-i}$, we may choose $X_{i}^\sharp=X_{n-i}$; its centralizer in $\bM_i=\bG^{n-i+1}$ is $\bM_{i+1}=\bG^{n-i}$ and as above we may choose $\phi_{i+1}'=\phi_{n-i+1}$.  Then $\tau_{i+1}=\tau_i(\phii{\phi}{i+1}')^{-1}|_H' = \prod_{j=0}^{n-i-1}\phii{\phi}{j}$, and this is a quasicharacter of depth $t_{i+1}=r_{n-i-1}$.

The algorithm terminates on $i=n+1$ since then $\tau_{n+1}=\triv$, which has depth zero. The corresponding output is $\Sigma_0=\Sigma$ and $\varphi=\tau_{n+1}^{-1}=\triv$, as we wished to show.
\end{proof}

Proposition~\ref{prop:resthenliftsingle} will be the key step in deducing that Theorem~\ref{th:lift} produces all $\Gamma$-stable character-data for $G$ (Proposition~\ref{prop:resthenlift}).

\subsection{Lifting the character-datum in sequence}\label{sec:completelift}

With Theorem \ref{T:liftone} in hand, we are now in a position to prove Theorem~\ref{th:lift}, the main result of this paper. 

Let $\bH = \bG^{[\Gamma]}$, and let $\Delta = (\vec{\bH},y,\vec{t},\vec{\xi})$ be a character-datum for $H$. That is, we have a sequence of twisted Levi subgroups
$$
\Gi{\bH}{0} \subsetneq \Gi{\bH}{1}\subsetneq \cdots \subsetneq \Gi{\bH}{n} \subseteq \Gi{\bH}{n+1}=\bH
$$
and for $0\leq i \leq n$, quasicharacters $\phii{\xi}{i}$ of $\Gi{H}{i}$ that are $\Gi{H}{i+1}$-generic of depth $t_i$. 

To construct a character-datum $\Sigma$ for $G$ whose corresponding semisimple character restricts to $\chara{\Delta}$, note that we cannot simply apply Theorem \ref{T:liftone} to all single-quasicharacter character-data $(\Gi{\bH}{i}\subseteq \Gi{\bH}{i+1},y,t_i,\phii{\xi}{i}), 0\leq i\leq n,$ 
independently as illustrated in Figure \ref{fig:liftcharsind}.  Doing so may produce Levi sequences that are incompatible with one another (that is, in the notation of the figure, $\Gi{\bG}{j,d_j} \not\subset \Gi{\bG}{j+1,0}$ for some $0\leq j\leq n-1$), as illustrated by Example~\ref{Eg:incompatible}. This naive lift could also produce quasicharacter sequences whose depths are not strictly increasing. Furthermore, even if neither obstruction occurs, the fixed-point datum of this lift, whose quasicharacter sequence is given by $(\phii{\varphi}{0}\phii{\xi}{0},\dots,\phii{\varphi}{n}\phii{\xi}{n})$, is unlikely to produce a semisimple character whose restriction to $\groupp{\Delta}$ coincides with  $\chara{\Delta}$, since the correction factors $\phii{\varphi}{i}$ generally have positive depth.

\begin{figure}[!htbp]
\begin{center}
    \begin{tikzpicture}
  \matrix (m) [matrix of math nodes,row sep=2em,column sep=0.5em,minimum width=0.1em]
  {
    &{} &{} &\Gi{\bG}{n,0}, & \dots, & \Gi{\bG}{n,d_n}  &{} &{} &  \phii{\phi}{n,0}, &\dots, & \phii{\phi}{n,d_n} \\
    & \Gi{\bG}{1,0}, &\dots, &\Gi{\bG}{1,d_1} &{} &{} &\phii{\phi}{1,0}, &\dots, &\phii{\phi}{1, d_1}\\
    \Gi{\bG}{0,0}, &\dots, &\Gi{\bG}{0,d_0} &{} &{} &\phii{\phi}{0,0}, &\dots, &\phii{\phi}{0,d_0}\\
    &\Gi{\bH}{0}, &\Gi{\bH}{1}, &\dots, &\Gi{\bH}{n} &{}  &\phii{\varphi}{0}\phii{\xi}{0}, &\phii{\varphi}{1}\phii{\xi}{1}, &\dots, &\phii{\varphi}{n}\phii{\xi}{n}\\};
       
\path[-stealth]
  (m-4-2) edge (m-3-1)
  (m-4-2) edge (m-3-2)
  (m-4-2) edge (m-3-3)
  (m-4-7) edge (m-3-6)
  (m-4-7) edge (m-3-7)
  (m-4-7) edge (m-3-8)
  (m-4-3) edge (m-2-2)
  (m-3-3) edge (m-2-3)
  (m-4-3) edge (m-2-4)
  (m-4-8) edge (m-2-7)
  (m-3-8) edge (m-2-8)
  (m-4-8) edge (m-2-9)
  (m-4-5) edge (m-1-4)
  (m-4-5) edge (m-1-5)
  (m-4-5) edge (m-1-6)
  (m-4-10) edge (m-1-9) 
  (m-4-10) edge (m-1-10)
  (m-4-10) edge (m-1-11)
  ;

\path
  (m-4-3) edge (m-3-3)
  (m-4-8) edge (m-3-8)
;

\end{tikzpicture}
\end{center}
\caption{\label{fig:liftcharsind}
Naive application of Theorem \ref{T:liftone} to the single-quasicharacter character-data $(\Gi{\bH}{i}\subseteq \Gi{\bH}{i+1}, y, t_i, \phii{\xi}{i})$, $0\leq i\leq n$.}
\end{figure}
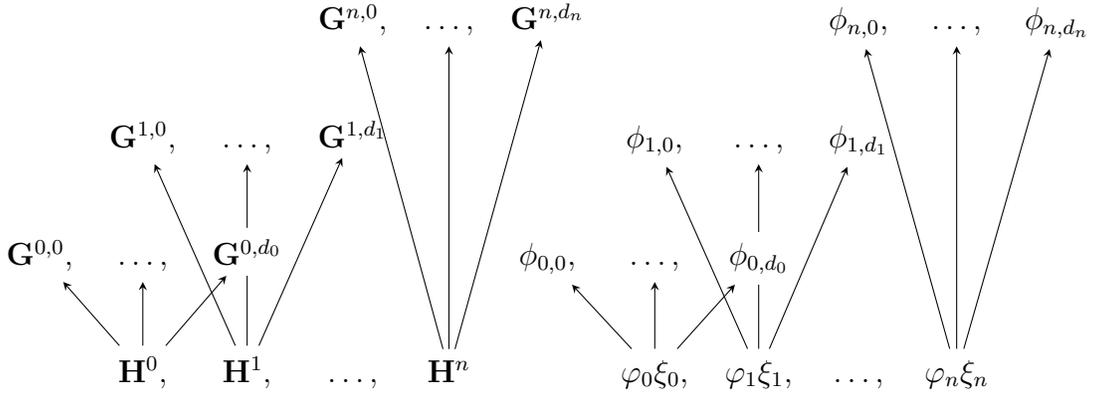

\begin{example} \label{Eg:incompatible}
Let $G=\GL(4,F)$, $H=G^{[\Gamma]}=\GL(2,F)\times\GL(2,F)$ and $\bH^0=\bT\times\bT$ as in Example~\ref{Eg:P:indstep}.  
Consider the character-datum for $H$ given by
$$
\Delta = (\bH^0\subsetneq \bH^1=\bH^2=\bH, y, (t_0,t_1), (\xi_0:=\eta\otimes\eta, \xi_1:=(\chi_+\circ\det) \otimes (\chi_-\circ\det))),
$$
where $\eta$ is a $\GL(2,F)$-generic quasicharacter of $T\cong E^\times$ of depth $t_0>0$ and $\chi_\pm$ are quasicharacters of $F^\times$ of depth $t_1>t_0$.  

Applying Theorem~\ref{T:liftone} with $(\bH',\xi,s)=(\bH^0,\xi_0,0)$ yields, as in Example~\ref{Eg:T:liftone}, the twisted Levi subgroup $\bG^{0,0}=\Res_{E/F}\GL(2)$ and the quasicharacter $\phi_0=\eta\circ \det$ of $G^{0,0}$ of depth $t_0$.  However, if  $\chi_+|_{F_{t_1}^\times}\neq \chi_-|_{F_{t_1}^\times}$, then applying Theorem~\ref{T:liftone} with $\bH'=\bH$ yields, as in Example~\ref{Eg:weird},  $\bG^{1,0}=\bH$ and the $G$-generic quasicharacter $\phi_1=\xi_1$ of depth $t_1$.   (More generally, this occurs whenever $\chi_+(\chi_-)^{-1}$ has depth greater than $t_0$.)
 These two do not combine to form a twisted Levi sequence in $\bG$ since $\bG^{0,0}$ and $\bG^{1,0}$ are incomparable with respect to inclusion.
\end{example}

The solution for obtaining a valid lift 
is to iteratively construct $\Sigma$ 
starting from the quasicharacter of $\Delta$ of largest depth. Let us sketch the idea before providing the detailed proof. 
The first step (referred to as Step $n$ in the proof of Theorem \ref{th:lift}) 
consists of applying Theorem \ref{T:liftone} to the tuple $(\Gi{\bH}{n} \subseteq \bG^{[\Gamma]}, y, t_n, \phii{\xi}{n})$ with $s = t_{n-1}$. This will produce quasicharacters of twisted Levi subgroups of $\bG$ whose $\Gamma$-fixed points are $\Gi{\bH}{n}$ and such that the product of the quasicharacters (upon restriction to $\Gi{H}{n}$) is a twist of $\phii{\xi}{n}$ by a correction factor, which is a quasicharacter $\phii{\varphi}{n}$ of $\Gi{H}{n}$ of depth at most $t_{n-1}$. We then twist $\phii{\xi}{n-1}$ by the inverse of this correction factor, $\phii{\varphi}{n}^{-1}$, 
and continue the iterative construction by applying Theorem \ref{T:liftone} to the tuple $(\Gi{\bH}{n-1}\subseteq (\Gi{\bG}{n,0})^{[\Gamma]}, y, t_{n-1}, \phii{\varphi}{n}^{-1}\phii{\xi}{n-1})$ with $s=t_{n-2}$. Continuing in this fashion, that is, twisting successive quasicharacters of the sequence $\vec{\xi}$ by the preceding correction factors as illustrated in Figure~\ref{fig:stepj}, 
we are able to construct a coherent lift $\Sigma$, such that $\Sigma^\Gamma$ is a refactorization of $\Delta$.  

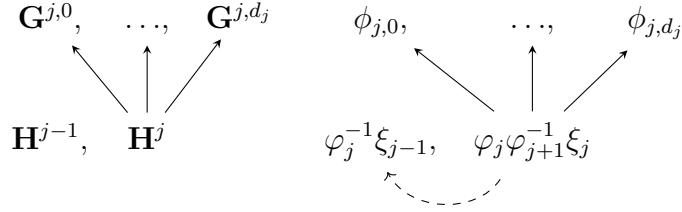
\begin{figure}[!h]
\begin{center}
    \begin{tikzpicture}
  \matrix (m) [matrix of math nodes,row sep=2em,column sep=0.5em,minimum width=0.1em]
    {
    \Gi{\bG}{j,0}, &{\dots}, &\Gi{\bG}{j,d_{j}} & &\phii{\phi}{j,0}, &{\dots}, &\phii{\phi}{j,d_{j}}\\
    \Gi{\bH}{j-1}, &\Gi{\bH}{j} &{}  & & \phii{\varphi}{j}^{-1}\phii{\xi}{j-1}, &\phii{\varphi}{j}\phii{\varphi}{j+1}^{-1}\phii{\xi}{j} &{}\\};
\path[-stealth]
    (m-2-2) edge (m-1-1)
    (m-2-2) edge (m-1-2)
    (m-2-2) edge (m-1-3)
    (m-2-6) edge (m-1-5)
    (m-2-6) edge (m-1-6)
    (m-2-6) edge (m-1-7)
  ;

  \draw[->, dashed, bend left=60] (m-2-6.west)+(0.5cm, -0.5cm) to (m-2-5);
\end{tikzpicture}
\end{center}
\caption{\label{fig:stepj}Illustration of Step $j$, for $j=n, n-1,n-2,\dots,1,0$, in the proof of Theorem \ref{th:lift}. The vertical arrows illustrate the application of Theorem \ref{T:liftone} to $(\Gi{\bH}{j}\subseteq (\Gi{\bG}{j+1,0})^{[\Gamma]}, y, t_j, \phii{\varphi}{j+1}^{-1}\phii{\xi}{j})$ with $s=t_{j-1}$ (setting $\varphi_{n+1}=\triv$ and $t_{-1} = 0$), which makes the correction factor $\phii{\varphi}{j}$ appear. If $j \geq 1$, the inverse of this correction factor then gets folded into $\phii{\xi}{j-1}$ in preparation for Step $j-1$, as illustrated by the dashed arrow.}
\end{figure}

\begin{proof}[Proof of Theorem \ref{th:lift}]
Let $\bH = \bG^{[\Gamma]}$, and let $\Delta = (\vec{\bH},y,\vec{t},\vec{\xi})$ be a character-datum for $H$, where
$$
\Gi{\bH}{0} \subsetneq \Gi{\bH}{1}\subsetneq \cdots \subsetneq \Gi{\bH}{n} \subseteq \Gi{\bH}{n+1}=\bH
$$
is a sequences of twisted Levi subgroups 
and for $0\leq i \leq n$, the quasicharacters $\phii{\xi}{i}$ of $\Gi{H}{i}$ are $\Gi{H}{i+1}$-generic of depth $t_i$.

We construct a character-datum $\Sigma$ for $G$ in $n+1$ iterative steps, indexed by $j$ decreasing from $n$ to $0$, as follows.

\underline{Initial Step ($j=n$):} Set $\tilde{\phii{\xi}{n}} = \phii{\xi}{n}$. Apply Theorem \ref{T:liftone} with $\bG = \bG$, $\bH' = \Gi{\bH}{n}$, $x=y$ $\xi = \tilde{\phii{\xi}{n}}$ and $s = t_{n-1}$ to produce 
a $\Gamma$-stable character-datum for $G$ together with a quasicharacter $\phii{\varphi}{n}$ of $\Gi{\bH}{n}$ of depth at most $t_{n-1}$.  Specifically, we have a twisted Levi sequence $$\Gi{\bG}{n,0}\subsetneq \Gi{\bG}{n,1}\subsetneq \cdots \subsetneq \Gi{\bG}{n,d_n}\subseteq\Gi{\bG}{n,d_n+1} = \bG$$  such that $(\Gi{\bG}{n,i})^{[\Gamma]} = \Gi{\bH}{n}$ for all $0\leq i\leq d_n$, as well as, for each $0\leq i\leq d_n$,  a $\Gi{G}{n,i+1}$-generic quasicharacter  $\phii{\phi}{n,i}$ of $\Gi{G}{n,i}$ of depth $r_{n,i}$, where $t_{n-1} < r_{n,0} < \cdots < r_{n,d_n} = t_n$. Furthermore, $\displaystyle \phii{\varphi}{n}\phii{\tilde{\xi}}{n} = \prod_{i=0}^{d_n}\phii{\phi}{n,i}|_{\Gi{H}{n}}.$

\underline{Step $j$, $j=n-1, n-2, \dots, 1,0$:} Set $\phii{\tilde{\xi}}{j} = (\phii{\varphi}{j+1}^{-1}|_{\Gi{H}{j}}) \phii{\xi}{j}$, which is a quasicharacter of $\Gi{H}{j}$. Since $\phii{\varphi}{j+1}$ is a quasicharacter of the larger group $\Gi{H}{j+1}$, the genericity of $\phii{\xi}{j}$ implies that $\phii{\tilde{\xi}}{j}$ is also $\Gi{H}{j+1}$-generic of depth $t_{j}$. Apply Theorem \ref{T:liftone} to $\bG = \Gi{\bG}{j+1,0}$, $\bH' = \Gi{\bH}{j}$, $x=y$, $\xi = \phii{\tilde{\xi}}{j}$ and $s=t_{j-1}$ (setting $t_{-1} = 0$) to produce a $\Gamma$-stable character-datum for $\Gi{G}{j+1,0}$ and a quasicharacter $\phii{\varphi}{j}$ of $\Gi{H}{j}$ of depth at most $t_{j-1}$.  That is, we have a twisted Levi sequence
$$
\Gi{\bG}{j,0}\subsetneq \Gi{\bG}{j,1}\subsetneq \cdots \subsetneq \Gi{\bG}{j,d_{j}}\subseteq \Gi{\bG}{j,d_{j}+1} =\Gi{\bG}{j+1,0}
$$ 
of $\Gi{\bG}{j+1,0}$ 
(and therefore of $\bG$) such that 
$(\Gi{\bG}{j,i})^{[\Gamma]} = \Gi{\bH}{j}$
for all $0\leq i\leq d_{j}$. 
In this case, since 
$(\Gi{\bG}{j,d_j+1})^{[\Gamma]} = \Gi{\bH}{j+1}$, the final inclusion is also strict. 
Moreover, for each $0\leq i\leq d_{j}$ we have   a 
$\Gi{G}{j,i+1}$-generic quasicharacter  $\phii{\phi}{j,i}$ of 
$\Gi{G}{j,i}$ of depth $r_{j,i}$.  Here, $t_{j-1} < r_{j,0} < \cdots < r_{j,d_{j}} = t_{j}$, which ensures that $r_{j,d_j}=t_j< r_{j+1,0}$.  Finally, we have $\displaystyle \phii{\varphi}{j}\phii{\tilde{\xi}}{j} = \prod_{i=0}^{d_{j}}\phii{\phi}{j,i}|_{\Gi{H}{j}}.$

At the conclusion of all these steps, we obtain:
$$\vec{\bG} = (\Gi{\bG}{0,0},\dots,\Gi{\bG}{0,d_0};\Gi{\bG}{1,0},\dots,\Gi{\bG}{1,d_1};\dots; \Gi{\bG}{n,0},\dots,\Gi{\bG}{n,d_n}),$$
a strictly increasing sequence of twisted Levi subgroups of $\bG$;
$$\vec{r} = (r_{0,0},\dots,r_{0,d_0};r_{1,0},\dots,r_{1,d_1};\dots;r_{n,0},\dots,r_{n,d_n}),$$
a strictly increasing sequence of positive real numbers such that for each $0\leq j\leq n$ we have $r_{j,d_j}=t_j$; and
$$\vec{\phi} = (\phii{\phi}{0,0},\dots,\phii{\phi}{0,d_0};\phii{\phi}{1,0},\dots,\phii{\phi}{1,d_1};\dots; \phii{\phi}{n,0},\dots,\phii{\phi}{n,d_n}),$$
a sequence of quasicharacters satisfying (CD4).  Thus $\Sigma = (\vec{\bG},x,\vec{r},\vec{\phi})$ is a character-datum for $G$; since the character-data at each step was $\Gamma$-stable, $\Sigma$ is $\Gamma$-stable as well.

It remains to prove that $\Sigma^\Gamma$ is a refactorization of $\Delta$, in the sense of Definition \ref{def:refactorization}.  We have $\Sigma^\Gamma = (\vec{\bH},x,\vec{t},\vec{\xi'})$, where $\vec{\xi'} = (\phii{\xi}{0}',\dots,\phii{\xi}{n}')$ with $\displaystyle\phii{\xi}{j}' = \prod_{i=0}^{d_j}\phii{\phi}{j,i}|_{\Gi{H}{j}} = \phii{\varphi}{j}\phii{\tilde{\xi}}{j}$.
Since by construction
$$\phii{\xi}{j}' = \phii{\varphi}{j}\phii{\tilde{\xi}}{j} = \begin{cases}
\phii{\varphi}{j}\phii{\varphi}{j+1}^{-1}\phii{\xi}{j} \text{ if } 0\leq j\leq n-1, \\
\phii{\varphi}{n}\phii{\xi}{n} \text{ if } j=n,
\end{cases}$$
it follows that, for all $0\leq i\leq n$, 
$$
\left(\prod_{j=i}^n\phii{\xi}{j}'\phii{\xi}{j}^{-1}\right)\Big|_{\Gi{H}{i}_{t_{i-1}^+}} = \phii{\varphi}{i}|_{\Gi{H}{i}_{t_{i-1}^+}} = 1,$$
as required. 
Since $\Sigma^\Gamma$ is a refactorization of $\Delta$, Lemma~\ref{lem:equivchardata} ensures that $\groupp{\Sigma^\Gamma} = \groupp{\Delta}$ and $\chara{\Sigma^\Gamma}=\chara{\Delta}$, which completes the proof of the theorem.
\end{proof}

Let us illustrate Theorem~\ref{th:lift} in the context of our continuing example by producing a lift in that case.

\begin{example}
Let $G=\GL(4,F)$, $H=G^{[\Gamma]}=\GL(2,F)\times\GL(2,F)$, and $\Delta$ as in Example~\ref{Eg:incompatible}, where $\chi_+(\chi_-)^{-1}$ has depth $r_{1,0}$ satisfying $t_0<r_{1,0}<t_1$. 
Here $n=1$ so the initial step $j=1$ yields
$$
\bG^{1,0}=\bH \subsetneq \bG^{1,1}=\bG^{1,2}=\bG,
$$
which satisfy $(\bG^{1,i})^{[\Gamma]}=\bH$ and quasicharacters $\phi_{1,0}$ of $H$, $\phi_{1,1}$ of $G$, and $\varphi_1$ of $H$ such that $\varphi_1\xi_1 = \phi_{1,0}(\phi_{1,1}|_{H})$.

At Step $j=0$, we begin with the quasicharacter $\tilde{\xi}_0=\varphi_1^{-1}|_{T\times T}(\eta\otimes \eta)$, which remains $H$-generic of depth $t_0$.  Applying Theorem~\ref{T:liftone} to $(\bG=\bH, \bH'=\bT\times\bT,\xi=\tilde{\xi}_0, s=0)$ thus simply yields $\bG^{0,0}=\bH^0$ and a character $\phi_{0,0}$ that equals $\tilde{\xi}_0$ up to a depth-zero twist $\varphi_0$.

Our result is thus the sequence $\bT\times \bT \subset \bH \subset \bG$ with the characters above.  More generally, all valid twisted Levi sequences drawn from $\bT\times \bT, \bH, \Res_{E/F}\GL(2), \bG$ do occur for some choices of the quasicharacters $\xi_0,\xi_1$ in $\Delta$.
\end{example}

We are now in a position to prove that every $\Gamma$-stable character-datum for $G$ arises as the lift of some character-datum for $G^{[\Gamma]}$.  We extract this observation from the proof of Theorem~\ref{th:lift}, together with Proposition~\ref{prop:resthenliftsingle}.

Let $\Delta$ be a character-datum for $H$ of length $n+1$.  To enumerate the many character-data for $G$ that are produced by Theorem~\ref{th:lift},  
we follow the $n+1$ iterative steps, which are indexed by $j$ decreasing from $n$ to $0$. Step $n$ consists of choosing $(\Sigma_n,\varphi_n)\in\mathrm{Lift}(\Delta_n)$, where $\Delta_n = (\Gi{\bH}{n}\subseteq \bH, y, t_n, \phii{\xi}{n})$. Every following step $j$ for $0\leq j\leq n-1$ then consists of choosing $(\Sigma_j, \varphi_j)\in \mathrm{Lift}(\Delta_j)$, where $\Delta_j = (\Gi{\bH}{j}\subseteq \Gi{\bH}{j+1},y, t_j, \phii{\varphi}{j+1}^{-1}\phii{\xi}{j})$. The resulting character-datum $\Sigma$ for $G$ is then the concatenation of $\Sigma_j, 0\leq j\leq n$. Letting $\mathrm{Lift}(\Delta)$  denote the set of character-data for $G$ obtained by varying the elements $(\Sigma_j,\varphi_j)\in \mathrm{Lift}(\Delta_j), 0\leq j\leq n$, we have the following proposition, which guarantees that our lift of character-data is compatible with the $\Gamma$-fixed point restriction. 

\begin{proposition}\label{prop:resthenlift}
Let $\Sigma$ be a $\Gamma$-stable character-datum for $G$. Then $\Sigma\in \mathrm{Lift}(\Sigma^\Gamma)$. In particular, every $\Gamma$-stable character-datum for $G$ is a lift of a character-datum for $H$.
\end{proposition}

The proof of this proposition consists of applying Proposition~\ref{prop:resthenliftsingle} at each step of the above construction.

\subsection{Relations to other constructions}\label{sec:otheraves}

We end this section by discussing connections between the constructions of this paper and other standard constructions in the literature, such as Howe factorizations \cite[Section 3.6]{Kaletha2019} and $G$-factorizations \cite[Section 5]{Murnaghan2011}.  

Recall that one of the key ingredients for proving our main lifting theorem was the construction of a character-datum that lifts a single quasicharacter (Theorem \ref{T:liftone}). Let us describe the output of Theorem \ref{T:liftone} when we set $s=0$ for particular choices of $\bH'\subseteq \bH$:
\begin{enumerate}[1)]
\item Setting $\bH'=\bS$ to be a maximal torus of $\bH$, then since $\Cent_\bG(\bS) = \bT\subseteq \bG^0$, then for any quasicharacter $\xi$ of $S$ the output of Theorem \ref{T:liftone} is a depth-zero quasicharacter $\varphi$ of $S$ together with a quasicharacter $\phi = \prod_{i=0}^n\phii{\phi}{i}|_{T}$ of $T$ such that $\phi|_S=\varphi \xi$.  This is
essentially a Howe factorization in the sense of \cite[Definition 3.6.2]{Kaletha2019} of some extension of $\xi$ to $T$. Indeed, by extending $\varphi$ to a depth-zero quasicharacter $\tilde{\varphi}$ of $T$, the quasicharacter $\tilde{\xi} = \tilde{\varphi}^{-1}\phi$ 
is an extension of $\xi$. Using \cite[Lemma 3.7.2]{Kaletha2019} to normalize $(\phii{\phi}{0},\dots,\phii{\phi}{n})$ if necessary, $(\tilde{\varphi}^{-1},\phii{\phi}{0},\dots,\phii{\phi}{n})$ is then a Howe factorization of $(T,\tilde{\xi})$.

\item Setting $\bH' = \bH$, Theorem \ref{T:liftone} asserts that for any quasicharacter $\xi$ of $H$ (which is $H$-generic by default), there exists a quasicharacter $\phi = \prod_{i=0}^n\phii{\phi}{i}|_{\Gi{G}{0}}$ of $\Gi{G}{0}$, where $\Gi{\bG}{0}$ is a twisted Levi subgroup of $\bG$, and a depth-zero quasicharacter $\varphi$ of $H$ such that $\varphi\xi = \phi|_H$. Thus more generally, given $\bH'\subseteq \bH$, the output of Theorem \ref{T:liftone} is a $G$-factorization, in the sense of \cite[Definition 5.1]{Murnaghan2011}, of an extension of some depth-zero twist of $\xi$ to $\Gi{G}{0}$.
\end{enumerate}

Extending this to character-data for $H$, we can relate the output of Theorem \ref{th:lift} to Howe factorizations as follows.

\begin{proposition}
Let $\bS$ be a maximal torus of $\bH = \bG^{[\Gamma]}$ defined over $F$, and set $\bT = \Cent_\bG(\bS)$. Let $\Delta = (\vec{\bH},y,\vec{t},\vec{\xi})$ be a character-datum for $H$ such that $\bS\subset \bH^0$ and $y\in \buil(S)$.  Let $\Sigma = (\vec{\bG},y,\vec{r},\vec{\phi})$ be a corresponding lift to $G$ obtained via Theorem~\ref{th:lift}. Then there exist characters $\chi$, $\theta$ of $T$, with $\chi$ of depth zero, such that $(\chi,\vec{\phi})$ is a (refactorization of a) Howe factorization of $(T,\theta)$.
\end{proposition}

\begin{proof}
Let $n+1$ and $d+1$ be the lengths of the character-data $\Delta$ and $\Sigma$, respectively. Define two characters $\theta_S$ and $\theta_T$ of $S$ and $T$, respectively, by
$$\theta_S = \prod_{i=0}^n\phii{\xi}{i}|_S \text{ and } \theta_T = \prod_{i=0}^d\phii{\phi}{i}|_T.$$ Upon normalizing the sequences $\vec{\xi}$ and $\vec{\phi}$ if necessary \cite[Lemma 3.7.2]{Kaletha2019}, we have that $(\phii{\xi}{-1} = 1,\vec{\xi})$ and $(\phii{\phi}{-1}=1,\vec{\phi})$ are Howe factorizations of $(S,\theta_S)$ and $(T,\theta_T)$, respectively. Furthermore, $\theta_T$ and $\theta_S$ agree on $S_{0^+}$, since $\chara{\Delta} = \chara{\Sigma}|_{\groupp{\Delta}}$ and $S_{0+}\subset \groupp{\Delta}$.  Extend the depth-zero character $\theta_S\theta_T^{-1}|_S$ trivially across $T_{0+}$ and choose an extension to a depth-zero quasicharacter $\chi$ of $T$. Then  $\theta = \chi\theta_T$ is an extension of $\theta_S$ to $T$ and $(\chi,\vec{\phi})$ is a Howe factorization of $(T,\theta)$.
 \end{proof}

 The  previous proposition suggests a different method for constructing a character-datum for $G$ from a character-datum for $H$. Indeed, beginning with a character-datum $\Delta = (\vec{\bH},y,\vec{t},\vec{\xi})$ for $H$ of length $n+1$, set $\theta = \prod_{i=0}^n\phii{\xi}{i}|_S$. Extend $\theta_S$ to a character $\theta_T$ of $T$. Follow \cite[Section 3.6]{Kaletha2019} to unfold the pair $(T,\theta_T)$ into a twisted Levi sequence $\vec{\bG}$ in $\bG$, and produce a Howe factorization $(\phii{\phi}{-1},\vec{\phi})$ of $(T,\theta_T)$ with corresponding depth sequence $(0,\vec{r})$. The resulting tuple $\Sigma = (\vec{\bG},y,\vec{r},\vec{\phi})$ is then a character-datum for $G$. This construction is summarized in Figure~\ref{fig:otherconstruction}.

\begin{figure}[!htbp]
\begin{center}
\begin{tikzcd}[row sep=4em, column sep=6em]
{\Delta=(\vec{\bH},y,\vec{t},\vec{\xi})}\arrow[swap]{d}{\displaystyle \theta_S := \prod_{i=0}^n\phii{\xi}{i}|_S} &{\Sigma = (\vec{\bG},y,\vec{r},\vec{\phi})} \\
{(S,\theta_S)}\arrow{r}{\parbox{2cm}{\centering \small extend}} &{(T,\theta_T)}\arrow[swap]{u}{\parbox{2cm}{\centering \small Howe \\ factorization}} 
\end{tikzcd}
\end{center}
\caption{\label{fig:otherconstruction}Construction of a character-datum for $G$ from a character-datum for $H$ using a Howe factorization.}
\end{figure}

However, not every choice of extension $\theta_T$ of $\theta_S$ will result in a $\Gamma$-stable character-datum $\Sigma$, 
as illustrated in the following example.  That is, for such choices of $\theta_T$, the corresponding semisimple characters $\chara{\Delta}$ and $\chara{\Sigma}$ are incompatible.

\begin{example}
Let $G=\GL(4,F)$, $p\neq 2$.  Choose $J=\left(\begin{smallmatrix}0&I_2\\-I_2&0\end{smallmatrix}\right)$; set $\gamma(g)={}^tJ({}^tg^{-1})J$ and $\Gamma = \{1,\gamma\}$. Then we have $H=G^{[\Gamma]}=\Sp(4,F)$.  Let $E/F$ be a quadratic extension and set $\bT_0 = \Res_{E/F}\bG_m$.  Then we can choose a maximal torus $\bS\cong \bT_0$ of $\bH$, embedded as a pair of block diagonal matrices $\diag(a, {}^ta^{-1})$ with $a\in \bT_0$, whose centralizer in $\bG$ is $\bT = \{\diag(a,b)\mid a,{}^tb\in \bT_0\} \cong \bT_0\times \bT_0$. 

Let $r>0$ and choose  quasicharacters $\eta$ of $F^\times$ and $\phi$ of $T_0$ of depth $r$ so that 
the quasicharacter $\theta_S$ of $S$ defined by $\theta_S(a,{}^ta^{-1}) = \phi(a)\eta(\Norm(a))^{-1}$ is $H$-generic of depth $r$.  
The quasicharacter
$$
\theta_T = (\phi, \eta\circ \Norm) : T_0\times T_0 \to \mathbb{C}^\times
$$
is an extension of $\theta_S$ to $T$.   It is not $G$-generic, but by construction $\theta_T$ is the restriction to $T$ of the $G$-generic quasicharacter $(\phi,\eta\circ\det)$ of $E^\times \times GL(2,F)$.  Thus the twisted Levi sequence corresponding to any Howe factorization of $\theta_T$ will include the group $\bT_0\times \GL(2)$, which is not $\Gamma$-stable.

In contrast, applying Theorem~\ref{T:liftone} to a corresponding \singlequasicharacter\ datum $\Delta=(\bS\subset \bH, x, r, \theta_S)$ yields $\Sigma = (\bT\subset \bG, x, r, \zeta)$ where $\zeta$ is a quasicharacter of $T$ realized by an element of $\LieS^*_{-r}$.  That is, $\zeta=(\xi,\xi^{-1})$ where $\xi$ a quasicharacter of $E^\times$, and furthermore, such that $\xi^2|_{E^\times_0}=\phi \otimes \eta^{-1}\circ \Norm|_{E^\times_0}$.  
\end{example}

The roadblock illustrated in this example is the lack of control one has when choosing $\theta_T$ arbitrarily:  the subspaces of $\LieT^*$ containing elements which realize the quasicharacters in a Howe factorization of  $\theta_T$ are not constrained by $\theta_S$.  In the construction of Theorem~\ref{th:lift}, this control is exerted on each factor of the Howe factorization through Theorem~\ref{T:charexists}.


\providecommand{\bysame}{\leavevmode\hbox to3em{\hrulefill}\thinspace}
\providecommand{\MR}{\relax\ifhmode\unskip\space\fi MR }
\providecommand{\MRhref}[2]{%
  \href{http://www.ams.org/mathscinet-getitem?mr=#1}{#2}
}
\providecommand{\href}[2]{#2}

\end{document}